\documentclass[11pt]{article}

\usepackage{amsmath, amsthm, amssymb}
\usepackage{graphicx}
\usepackage{psfrag}
\usepackage{epsfig}
\usepackage{epstopdf}
\usepackage{color}

\newtheorem{theorem}{Theorem}[section]
\newtheorem{corollary}{Corollary}[section]
\newtheorem{lemma}{Lemma}[section]
\newtheorem{remark}{Remark}[section]

\usepackage{fullpage}
\def\ltt{\lesssim}
\def\gtt{\gtrsim}

\newcommand{{\Reals}}{\mathbb{R}}

\def\bsigma{\boldsymbol\sigma}

\usepackage{cite}

\numberwithin{equation}{section}
\allowdisplaybreaks

\date{\empty}
\begin{document}

\title{A non-iterative domain decomposition method for the interaction between a fluid and a thick structure}
\author{
{Anyastassia Seboldt}
\thanks{Department of Applied and Computational Mathematics and Statistics, University
of Notre Dame, Notre Dame, IN 46556, USA. Email: {aseboldt@nd.edu}. }
\and 
{Martina Buka\v{c}}
\thanks
{Department of Applied and Computational Mathematics and Statistics, University
of Notre Dame, Notre Dame, IN 46556, USA. Email: {mbukac@nd.edu}. (Corresponding author.)}
}


\maketitle


  \begin{abstract}
  This work focuses on the development and analysis of a partitioned numerical method for moving domain, fluid-structure interaction problems. We model the fluid using incompressible Navier-Stokes equations, and the structure using linear elasticity equations. We assume  that the structure is thick, i.e., described in the same dimension as the fluid.  
We propose a non-iterative, domain decomposition method where the fluid and the structure sub-problems are solved separately. The method is based on generalized Robin boundary conditions, which are used in both fluid and structure sub-problems. Using energy estimates, we show that the proposed method applied to a  moving domain problem is unconditionally stable. We also analyze the convergence of the method and show $\mathcal{O}(\Delta t^\frac12)$ convergence in time and optimal convergence in space. Numerical examples are used to demonstrate the performance of the method. In particular, we explore the relation between the combination parameter used in the derivation of the generalized Robin boundary conditions and the accuracy of the scheme. We also compare the performance of the method to a  monolithic solver. 
  \end{abstract}

\section{Introduction}

Fluid-structure interaction (FSI) problems arise in many applications, such as aerodynamics, hemodynamics and geomechanics. They are used to predict flow properties in patient-specific arterial geometries, microfluidic devices and in the design of many industrial components.   
FSI problems are moving domain problems,  characterized by highly non-linear coupling between fluid flow and structure deformation. As a result, the development of robust numerical algorithms is a subject of intensive research.

The solution strategies for FSI problems can be classified as monolithic and  partitioned
methods. In monolithic algorithms~\cite{bazilevs2008isogeometric,deparis2003acceleration,gerbeau2003quasi,nobile2001numerical,gee2011truly,ryzhakov2010monolithic,hron2006monolithic,bathe2004finite}, the coupling conditions are imposed implicitly and the entire coupled problem is solved as one system of algebraic equations. However, they may require  long computational time, large memory allocation and
 well-designed preconditioners~\cite{gee2011truly,badia2008modular,heil2008solvers}.
In partitioned methods~\cite{degroote2008stability,farhat2006provably,bukavc2012fluid,badia2009robin,BorSunMulti,Fernandez2012incremental,fernandez2013fully,nobile2008effective,hansbo2005nitsche,lukavcova2013kinematic,banks2014analysis,banks2014analysis2,oyekole2018second,bukavc2016stability}, the fluid flow and structure deformation are solved separately as smaller and better conditioned sub-problems, which reduces the computational cost.
However, they often suffer from numerical instabilities, which makes the design and analysis of stable and efficient partitioned schemes
challenging even for simplified, linear problems.

 The design of partitioned algorithms is especially challenging in blood flow applications due to numerical instabilities known as \textit{the added mass effect}~\cite{causin2005added}, which are manifested when the fluid and structure have comparable densities. Furthermore, design of non-iterative, partitioned methods is particularly difficult when the dimension of the solid domain is the same as the dimension of the fluid domain. When the structure is thin, i.e., described by a lower-dimensional model, it serves as a fluid-structure interface with mass, which is exploited in the design of many partitioned methods~\cite{oyekole2018second,bukavc2016stability,Fernandez2012incremental,lukavcova2013kinematic} where parts of the structure equation are used as a Robin boundary condition for the fluid problem. However, when the structure is thick, no additional mass is present at the fluid-structure interface, which makes the design of stable, non-iterative partitioned algorithms especially challenging. 
 
It is well-known that classical, Dirichlet-Neumann partitioned methods are unconditionally unstable when fluid and structure have comparable densities~\cite{causin2005added}, which can be resolved by  sub-iterating between fluid and structure sub-problems within each time step. As an alternative to the Dirichlet-Neumann approach, which can exhibit convergence issues, Robin-Dirichlet, Robin-Neumann, or Robin-Robin  methods were designed in~\cite{nobile2008effective,badia2009robin,badia2008fluid,gerardo2010analysis,degroote2011similarity}. In the design of these methods,  the coupling conditions are linearly combined to obtain the generalized Robin interface conditions, which are then used in the fluid and/or structure  sub-problems. 
We also mention the fictitious-pressure and fictitious-mass algorithms proposed in~\cite{baek2012convergence,yu2013generalized}, in which  the added mass effect is accounted for by incorporating additional terms into governing equations. However, algorithms proposed in~\cite{nobile2008effective,badia2009robin,badia2008fluid,gerardo2010analysis,degroote2011similarity,baek2012convergence,yu2013generalized} still require sub-iterations between the fluid and the structure sub-problems in order to achieve stability.

A different partitioned scheme  was proposed in~\cite{burman2009stabilization,burman2013unfitted}, where the fluid-structure coupling conditions are imposed using Nitsche's penalty method~\cite{hansbo2005nitsche} and some terms are time-lagged to uncouple the fluid and solid sub-problems. It was shown that the scheme is stable under a CFL condition if a weakly consistent stabilization term
that includes pressure variations at the interface is added. The authors show that the rate of convergence in time is sub-optimal, which is then corrected by proposing a few defect-correction sub-iterations. A non-iterative, partitioned algorithm based on the so-called added-mass partitioned Robin conditions was proposed  in~\cite{banks2014analysis2}.  It was shown that the algorithm is stable  under a condition on  the time step, which depends on the structure parameters. Even though the authors do not derive the convergence rates, their numerical results indicate that the scheme is second-order accurate in time. 
A generalized Robin-Neumann explicit coupling scheme based on an interface operator accounting for the
solid inertial effects within the fluid has been proposed in~\cite{fernandez2015generalized}. The scheme has been analyzed on a linear FSI problem and shown to be stable under a time-step condition. 
In our previous work~\cite{bukavc2014modular}, we developed a partitioned scheme for FSI with a thick, linearly viscoelastic structure based on an operator-splitting approach. However, the assumption that the structure is viscoelastic was necessary in the derivation of the scheme, and the solid viscosity was solved implicitly with the fluid problem. Furthermore, the scheme was shown to be stable only under a condition on the time step~\cite{bukavc2016stability}.

In this work, we propose a partitioned, loosely-coupled  method for FSI problems with thick structures. As opposed to the previous work, the method presented here is unconditionally stable, and  sub-iterations or stabilization terms are not needed to achieve stability. 
Furthermore,  a moving domain problem  was considered in the stability analysis. 
The fluid is modeled using the Navier-Stokes equations for an incompressible, viscous fluid, and the structure using the equations of linear elasticity. The deformation of the fluid mesh is treated using the Arbitrary Lagrangian-Eulerian approach (ALE)~\cite{hughes1981lagrangian,donea1983arbitrary,nobile2001numerical}, where the fluid mesh is allowed to deform matching the deformation of the structural domain. 
The proposed partitioned method is based on generalized Robin boundary conditions, which are formulated in a novel way. Unconditional stability is shown  on a  moving domain, semi-discrete problem using energy estimates. The proposed method is discretized in space and implemented using the finite element method. We  preform error analysis of the fully discrete method on a linearized problem and show that the scheme exhibits $\mathcal{O}(\Delta t^\frac12)$ convergence in time and optimal convergence in space. The relation between the combination parameter used in the formulation of generalized Robin boundary conditions and the accuracy of the method is explored in the numerical examples. We also compare our method to an implicit scheme on a benchmark problem under realistic parameters in blood flow modeling. 

This paper is organized as follows. The non-linear FSI problem is presented in Section 2, and the proposed numerical scheme is presented in Section 3. Stability analysis
is performed in Section 4 and error analysis is performed in Section 5. Numerical examples are presented in Section 6. Conclusions are drawn in Section 7.



\section{Mathematical model}

We are interested in modeling fluid flow in a deformable channel, where the channel walls represent an elastic structure. We assume that the fluid is viscous and incompressible, that the structure is linearly elastic, and that the fluid and structure are both described in two-dimensional domains. The fluid and structure are two-ways coupled, resulting in a non-linear, moving domain problem.

\subsection{Computational domains and mappings}

We denote the reference fluid domain by $\hat{\Omega}_F$ and the reference structure domain by $\hat{\Omega}_S$ (see Figure~\ref{domain}).  The fluid and structure domains at time $t$ are denoted by $\Omega_F(t)$ and $\Omega_S(t)$, respectively. 
\begin{figure}[ht]
\centering{
\includegraphics[scale=0.6]{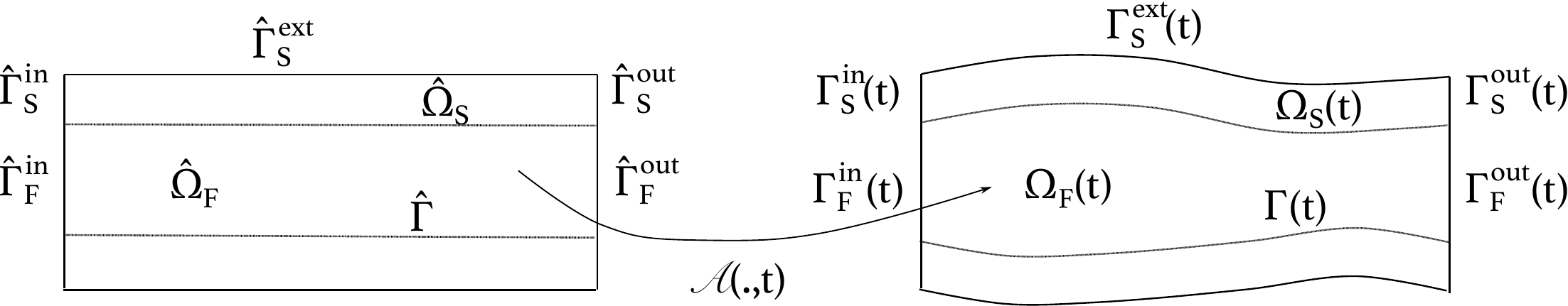}
}
\caption{Left: Reference domain $\hat{\Omega}_F \cup \hat{\Omega}_S$. Right: Deformed domain $\Omega_F(t) \cup \Omega_S(t).$}
\label{domain}
\end{figure}

We assume that the structure equations are given in a Lagrangian framework, with respect to the reference domain $\hat{\Omega}_S$.
The fluid equations will be described in the ALE formulation. To track the deformation of the fluid domain in time, we introduce a smooth, invertible, ALE mapping ${\mathcal{A}}: \hat{\Omega}_F \times [0,T] \rightarrow \Omega_F(t)$ given by
\begin{equation*}
{\mathcal{A}} ({\boldsymbol X},t)=  {\boldsymbol X} + {\boldsymbol \eta}_F({\boldsymbol X},t), \quad \; \textrm{for all } {\boldsymbol X} \in \hat{\Omega}_F, t \in [0,T],
\label{ale}
\end{equation*}
where ${\boldsymbol \eta}_F$ denotes the  displacement of the fluid domain. We assume that ${\boldsymbol \eta}_F$  equals the structure displacement on $\hat\Gamma$, and is arbitrarily extended into the fluid domain $\hat{\Omega}_F$~\cite{langer2018numerical}.
We denote the fluid deformation gradient by ${\boldsymbol F} = {\nabla}\mathcal{A}$ and its determinant by ${J}$.

\subsection{Fluid sub-problem}
To model the fluid flow, we use the Navier-Stokes equations in the ALE formulation~\cite{langer2018numerical,multilayered,thick}, given as follows:
\begin{align}\label{NSale1}
& \rho_F \left(  \partial_t \boldsymbol{v} |_{\hat{\Omega}_F}+ (\boldsymbol{v}-\boldsymbol{w}) \cdot \nabla \boldsymbol{v} \right)  = \nabla \cdot \boldsymbol\sigma_F(\boldsymbol v, p) + \boldsymbol f_F&  \textrm{in}\; \Omega_F(t)\times(0,T), \\
 \label{NSale2}
&\nabla \cdot \boldsymbol{v} = 0 & \textrm{in}\; \Omega_F(t)\times(0,T),
\end{align}
where $\boldsymbol{v}$ is the fluid velocity, $\boldsymbol w= \partial_t \boldsymbol x|_{\hat{\Omega}_F} = \partial_t \mathcal{A} \circ \mathcal{A}^{-1}$ is the domain velocity, $\rho_F$ is the fluid density,  $\boldsymbol\sigma_F$ is the fluid stress tensor and $\boldsymbol f_F$ is the forcing term. 
For a Newtonian fluid, the stress tensor is given by $\boldsymbol\sigma_F(\boldsymbol v,p) = -p \boldsymbol{I} + 2 \mu_F \boldsymbol{D}(\boldsymbol{v}),$ where  $p$ is the fluid pressure,
$\mu_F$ is the fluid viscosity and  $\boldsymbol{D}(\boldsymbol{v}) = (\nabla \boldsymbol{v}+(\nabla \boldsymbol{v})^{T})/2$ is the strain rate tensor. Notation $ \partial_t \boldsymbol{v} |_{\hat{\Omega}_F}$ denotes  the Eulerian description of the ALE field $\partial_t {\boldsymbol{v}} \circ \mathcal{A}$~\cite{formaggia2010cardiovascular}, {\emph i.e.},
\begin{equation*}
 \partial_t \boldsymbol{v}(\boldsymbol x,t) |_{\hat{\Omega}_F} = \partial_t \boldsymbol{v}(\mathcal{A}^{-1}(\boldsymbol x,t),t).
\end{equation*}

We denote the inlet and outlet  of the fluid domain by  $\Gamma_F^{in}(t)$ and $\Gamma_F^{out}(t)$, respectively.
At the inlet and outlet sections, we prescribe Neumann boundary conditions:
\begin{align}
& \boldsymbol {\sigma}_F \boldsymbol{ n}_{F} = -p_{in} (t) \boldsymbol{n}_{F}  &  \textrm{on} \; \Gamma_{F}^{in}(t) \times (0,T), \label{inlet}  \\
& \boldsymbol {\sigma}_F \boldsymbol{ n}_{F} = -p_{out}(t)  \boldsymbol{n}_{F}  &  \textrm{on} \; \Gamma_{F}^{out}(t) \times (0,T), \label{outlet} 
\end{align}
where $\boldsymbol n_F$ is the outward unit normal to the deformed fluid domain.
We will also consider the dynamic pressure inlet and outlet data:
\begin{align}\
& \displaystyle{p+\frac{\rho_F}{2}|\boldsymbol v|^2}=p_{in}(t) &  \textrm{on} \; \Gamma_{F}^{in}(t) \times (0,T), \label{inlet2} \\
& \displaystyle{p+\frac{\rho_F}{2}|\boldsymbol v|^2}=p_{out}(t) &  \textrm{on} \; \Gamma_{F}^{out}(t) \times (0,T), \label{outlet2} \\
& \boldsymbol v \times \boldsymbol{n}_F = 0  &  \textrm{on} \; \Gamma_{F}^{in}(t) \cup \Gamma_F^{out}(t) \times (0,T). \label{dp3} 
\end{align}
Here, the fluid flow is driven by a prescribed dynamic pressure drop, 
and the flow enters and leaves the fluid domain orthogonally to the inlet and outlet boundary. 
While  Neumann boundary conditions~\eqref{inlet}-\eqref{outlet} are more convenient to use in numerical simulations, dynamic pressure boundary conditions~\eqref{inlet2}-\eqref{dp3} are used to derive the energy estimates of the fluid problem in a moving domain and in the stability analysis.

\subsection{Structure sub-problem}

To model the elastic structure, we use the elastodynamics equations written in the first order form as
\begin{align}
& \partial_{t} {\boldsymbol \eta}  =  \boldsymbol \xi &  \textrm{in}\; \hat{\Omega}_S\times(0,T), 
\\
&{\rho}_S \partial_{t} {\boldsymbol \xi}  =  {\nabla} \cdot \boldsymbol \sigma_S(\boldsymbol \eta) + \boldsymbol f_S&  \textrm{in}\; \hat{\Omega}_S\times(0,T), 
\end{align}
where $\boldsymbol{\eta}$ is the structure displacement,   $\boldsymbol{\xi}$ is the structure velocity, ${\rho}_S$ is the structure density, ${\boldsymbol \sigma_S}$ is the solid stress tensor and $\boldsymbol f_S$ is the volume force applied to the structure. We assume that the deformations are small and use the Saint-Venant Kirchhoff elastic model, given as
\begin{align*}
 \boldsymbol \sigma_S(\boldsymbol \eta) = 2 \mu_S \boldsymbol D(\boldsymbol \eta) + \lambda_S (\nabla \cdot \boldsymbol \eta) \boldsymbol I,
\end{align*}
where $\mu_S$ and $\lambda_S$ are Lam\'e constants.
We assume that the structure is fixed at the inlet and outlet boundaries:
\begin{equation}\label{homostructure1}
 {\boldsymbol \eta} = 0 \quad \textrm{on} \;\; \hat{\Gamma}_S^{in} \cup \hat{\Gamma}_S^{out} \times(0,T),
\end{equation}
and that the external structure boundary $\hat{\Gamma}_S^{ext}$ 
is exposed to zero external ambient pressure:
\begin{equation}\label{homostructure2}
{\boldsymbol \sigma_S}  {\boldsymbol{ n}}_S =  0 \quad \textrm{on}  \;\; \hat{\Gamma}_S^{ext} \times (0,T),
 \end{equation}
 where ${\boldsymbol n}_S$ is the outward normal to the reference structure domain.

\subsection{The  coupled FSI problem}
To couple the fluid and structure sub-problems, we prescribe the kinematic and dynamic coupling conditions~\cite{langer2018numerical,multilayered} given as follows:

\noindent \textbf{Kinematic coupling condition} describes the continuity of  velocity at the fluid-structure interface (no-slip):
\begin{equation}
{\boldsymbol{v}} \circ \mathcal{A}= {\boldsymbol \xi}  \;\; \; \textrm{on} \; \hat{\Gamma} \times (0,T). \label{kinematic}
\end{equation}  

\noindent \textbf{Dynamic coupling condition} describes the continuity of stresses at the fluid-structure interface due to the action-reaction principle. The condition reads:
\begin{equation}
 {J} \boldsymbol \sigma_F \boldsymbol{F}^{-T} \boldsymbol n_F +  {\boldsymbol \sigma_S } {\boldsymbol n}_S=0 \;\;\; \textrm{on} \; \hat\Gamma \times (0,T). \label{dynamic}
 \end{equation}
Hence, the fully-coupled fluid-structure interaction problem is given by:
\begin{align}\label{fsi1}
& \rho_F \left( \partial_t \boldsymbol{v} |_{\hat{\Omega}_F}+ (\boldsymbol{v}-\boldsymbol{w}) \cdot \nabla \boldsymbol{v} \right)  = \nabla \cdot \boldsymbol\sigma_F(\boldsymbol v, p) &  \textrm{in}\; \Omega_F(t)\times(0,T), \\
&\nabla \cdot \boldsymbol{v} = 0 & \textrm{in}\; \Omega_F(t)\times(0,T), 
\label{fsi11}\\
& \partial_{t} {\boldsymbol \eta}  =  \boldsymbol \xi &  \textrm{in}\; \hat{\Omega}_S\times(0,T), 
\\
&{\rho}_S \partial_{t} {\boldsymbol \xi}  =  {\nabla} \cdot \boldsymbol \sigma_S(\boldsymbol \eta) &  \textrm{in}\; \hat{\Omega}_S\times(0,T), 
\label{fsi12}
\\
& {\boldsymbol{v}} \circ \mathcal{A}= {\boldsymbol \xi} & \textrm{on} \; \hat{\Gamma} \times (0,T), \label{coupling_noslip}\\
&   {J} \boldsymbol \sigma_F \boldsymbol{F}^{-T} \boldsymbol n_F+  {\boldsymbol \sigma_S } {\boldsymbol n}_S =0 & \textrm{on} \; \hat\Gamma \times (0,T).\label{fsi2}
\end{align}
To update the fluid domain, we extend the solid displacement at the interface  using the harmonic extension, which is a common choice of the extension operator~\cite{badia}.
The fluid domain and domain velocity are determined, respectively, by
\begin{gather*}
\Omega_F(t) = \mathcal{A}(\hat{\Omega}_F, t),
\quad
\boldsymbol w = \partial_t  \mathcal{A} \circ \mathcal{A}^{-1}.
\end{gather*}

Initially, the fluid and the structure are assumed to be at rest, with zero displacement from the reference configuration.

\subsection{The weak formulation of the coupled problem}
Given an open set $S$, we  consider the usual Sobolev spaces $H^k(S)$, with $k \geq 0$. 
For all $t \in [0,T]$ we introduce the following 
functional spaces:
\begin{align*}
&V^F(t) =\left\{ \boldsymbol \phi: \Omega_F(t) \rightarrow \mathbb{R}^2 \ | \ \boldsymbol \phi = \hat{\boldsymbol \phi} \circ \mathcal{A}^{-1}, \  \hat{\boldsymbol \phi} \in (H^1(\hat{\Omega}_F))^2   \right\},  \\
&V^{F,0}(t) =\left\{ \boldsymbol \phi \in V^F(t) \ | \  \boldsymbol \phi \times \boldsymbol n =0  \; \; \textrm{on} \; \Gamma_{F}^{in} \cup \Gamma_F^{out} \right\},  \\
&  Q^F(t)=  \left\{  \psi: \Omega_F(t) \rightarrow \mathbb{R} \ | \ \psi = \hat{ \psi} \circ \mathcal{A}^{-1}, \  \hat{ \psi} \in L^2(\hat{\Omega}_F) \right\},  \\
& V^S = \left\{ {\boldsymbol {\zeta}}: \hat{\Omega}_S \rightarrow \mathbb{R}^2 \ | \ {\boldsymbol {\zeta}} \in (H^1(\hat{\Omega}_S))^2 , \; {\boldsymbol {\zeta}}=0 \; \textrm{on} \; \hat{\Gamma}_S^{in} \cup \hat{\Gamma}_S^{out} \right\}, \\
& V^{FSI}(t) =  \left\{ (\boldsymbol \phi, {\boldsymbol \zeta}) \in V^{F,0}(t) \times V^S \ | \ \boldsymbol \phi  = {\boldsymbol \zeta} \circ \mathcal{A}^{-1} \; \textrm{on} \; \Gamma(t) \right\}.
\end{align*}
We define the following bilinear forms associated with the fluid and structure problems:
\begin{align}
&a_F(\boldsymbol v, \boldsymbol \phi) = 2 \mu_F \int_{\Omega_F(t)} \boldsymbol D(\boldsymbol v) : \boldsymbol D(\boldsymbol \phi) d \boldsymbol x, \quad  \forall \boldsymbol v, \boldsymbol \phi \in V^F(t),
\\
&b_F(\boldsymbol v,  \psi) =  \int_{\Omega_F(t)} \nabla \cdot \boldsymbol v \psi d \boldsymbol x, \quad  \forall \boldsymbol v \in V^F(t), \psi \in Q^F(t),
\\
&a_S(\boldsymbol \eta, \boldsymbol \zeta) = 2 \mu_S \int_{\hat{\Omega}_S} \boldsymbol D(\boldsymbol \eta) : \boldsymbol D(\boldsymbol \zeta) d \boldsymbol x+ \lambda_S \int_{\hat{\Omega}_S} (\nabla \cdot \boldsymbol \eta)(\nabla \cdot \boldsymbol \zeta) d \boldsymbol x, \quad   \forall \boldsymbol \eta, \boldsymbol \zeta \in V^S.
\end{align}
We also define norm $\|  \cdot \|_S$ associated with the bilinear form $a_S(\cdot, \cdot)$ as
\begin{gather}
\| \boldsymbol \eta \|_S = \left(a_S(\boldsymbol \eta,\boldsymbol \eta)\right)^{\frac12}.
\end{gather}

The weak formulation of the coupled fluid-structure interaction problem~\eqref{fsi1}-\eqref{fsi2} with boundary conditions~\eqref{inlet2}-\eqref{dp3} and~\eqref{homostructure1}-\eqref{homostructure2} is given as follows: Find $(\boldsymbol v, \boldsymbol \xi) \in V^{FSI}(t), p \in Q^F(t)$ and $ \boldsymbol \eta \in V^S$ such that $\partial_t \boldsymbol \eta = \boldsymbol \xi$ and 
\begin{gather*}
\rho_F \int_{\Omega_F(t)} \partial_t \boldsymbol v|_{\hat{\Omega}_F} \cdot \boldsymbol \phi d\boldsymbol{x} 
+\rho_F \int_{\Omega_F(t)} \left((\boldsymbol v-\boldsymbol w) \cdot \nabla \right) \boldsymbol v \cdot \boldsymbol \phi  d\boldsymbol{x} 
+2 \mu_F \int_{\Omega_F(t)} \boldsymbol
D(\boldsymbol v) : \boldsymbol D (\boldsymbol \phi) d \boldsymbol x
\notag \\
- \int_{\Omega_F(t)} p \nabla \cdot \boldsymbol \phi d\boldsymbol{x} 
+ \int_{\Omega_F(t)} q \nabla \cdot \boldsymbol v d\boldsymbol{x} 
+  {\rho}_S \int_{\hat\Omega_S}  \partial_{t} {\boldsymbol \xi} \cdot {\boldsymbol \zeta} d {\boldsymbol X}
+  2 \mu_S \int_{\hat\Omega_S}   D(\boldsymbol \eta) : \boldsymbol D (\boldsymbol \zeta) d \boldsymbol X
\notag \\
+  \lambda_S \int_{\hat\Omega_S}   (\nabla \cdot \boldsymbol \eta) (\nabla \cdot \boldsymbol \zeta) d \boldsymbol X
= 
- \int_{\Gamma_F^{in}} p_{in} \boldsymbol \phi \cdot \boldsymbol{n}_F dx
- \int_{\Gamma_F^{out}} p_{out} \boldsymbol \phi \cdot \boldsymbol{n}_F dx
\notag 
\\
+\frac{\rho_F}{2} \int_{\Gamma_F^{in} \cup \Gamma_F^{out}} |\boldsymbol v |^2 \boldsymbol \phi \cdot \boldsymbol{n}_F dx,
\end{gather*}
for all $(\boldsymbol \phi, {\boldsymbol \zeta}) \in V^{FSI}(t), q \in Q^F(t)$.

To  derive the energy of the coupled FSI problem, we take $\boldsymbol \phi = \boldsymbol v, q=p$ and $\boldsymbol \zeta=\boldsymbol \xi$.
We transform $\int_{\Omega_F(t)} \rho_F   \partial_t {\boldsymbol{v}} |_{\hat{\Omega}_F} \cdot \boldsymbol v d\boldsymbol x$ on the reference domain $\hat{\Omega}_F$ as follows:
\begin{align*}
\int_{\Omega_F(t)}  \rho_F   \partial_t {\boldsymbol{v}} |_{\hat{\Omega}_F} \cdot \boldsymbol v d\boldsymbol x
& 
=
 \int_{\hat{\Omega}_F} \rho_F   J \partial_t \left({\boldsymbol{v} \circ \mathcal{A}}\right)  \cdot \left( {\boldsymbol v}  \circ \mathcal{A} \right) d\hat{\boldsymbol x}
 \\
 & 
 = 
\frac12 \int_{\hat{\Omega}_F}  \rho_F   \partial_t \left( J | {\boldsymbol{v}} \circ \mathcal{A} |^2 \right)  d\hat{\boldsymbol x}
-\frac12 \int_{\hat{\Omega}_F}  \rho_F   \partial_t   J  | {\boldsymbol{v}} \circ \mathcal{A}|^2  d\hat{\boldsymbol x}.
\end{align*}
Using the Euler expansion formula,
\begin{align}
\partial_t J|_{\hat{\Omega}_F} = J {\nabla} \cdot {\boldsymbol w},
\label{euleref}
\end{align}
 we have
\begin{align*}
\int_{\Omega_F(t)}  \rho_F   \partial_t {\boldsymbol{v}} |_{\hat{\Omega}_F} \cdot \boldsymbol v d\boldsymbol x
& 
=
\frac12 \int_{\hat{\Omega}_F}  \rho_F   \partial_t \left( J | {\boldsymbol{v} \circ \mathcal{A} }|^2 \right)   d\hat{\boldsymbol x}
-\frac12 \int_{\hat{\Omega}_F}  \rho_F    J {\nabla} \cdot \left( {\boldsymbol w} \circ \mathcal{A} \right) | {\boldsymbol{v}\circ \mathcal{A} }|^2   d\hat{\boldsymbol x}
   \\
 & 
=
\frac12 \frac{d}{dt} \int_{\hat{\Omega}_F}  \rho_F     J | {\boldsymbol{v} \circ \mathcal{A}}|^2  d\hat{\boldsymbol x}
-\frac12 \int_{\hat{\Omega}_F}  \rho_F    J {\nabla} \cdot \left( {\boldsymbol w} \circ \mathcal{A} \right) | {\boldsymbol{v}\circ \mathcal{A} }|^2   d\hat{\boldsymbol x}
     \\
 & 
 = 
\frac12 \frac{d}{dt} \int_{{\Omega_F(t)}}  \rho_F      | {\boldsymbol{v}}|^2   d{\boldsymbol x}
-\frac12 \int_{{\Omega_F(t)}}  \rho_F   {\nabla} \cdot {\boldsymbol w} | {\boldsymbol{v}}|^2   d{\boldsymbol x}.
\end{align*}
To handle the convective term,  after integration by parts and taking into account $\nabla \cdot \boldsymbol v=0$, we have
\begin{align}
\rho_F\int_{\Omega_F(t)}\left((\boldsymbol{v}-\boldsymbol{w}) \cdot \nabla\right) \boldsymbol{v} \cdot \boldsymbol v d\boldsymbol x 
& =
\frac{\rho_F}{2} \int_{\Omega_F(t)} \nabla \cdot \boldsymbol w  |\boldsymbol v|^2 d\boldsymbol x 
 +\frac{\rho_F}{2}  \int_{\Gamma(t)} \left((\boldsymbol v - \boldsymbol w) \cdot \boldsymbol n_F \right) |\boldsymbol v|^2 dS
\notag
\\
&
+\frac{\rho_F}{2}  \int_{\Gamma_F^{in}(t) \cup \Gamma_F^{out}(t)} \left((\boldsymbol v - \boldsymbol w) \cdot \boldsymbol n_F \right) |\boldsymbol v|^2 dS.
\notag
\end{align}
Since $\boldsymbol w = \boldsymbol u$ on $\Gamma(t)$ and $\boldsymbol w = \boldsymbol 0$ on $\Gamma_F^{in} \cup \Gamma_F^{out}$,  the following energy equality holds:
\begin{gather*}
\frac{\rho_F}{2} \frac{d}{dt}    \|   {\boldsymbol{v}}  \|^2_{L^2(\Omega_F(t))}
+2 \mu_F \| \boldsymbol D(\boldsymbol v) \|^2_{L^2(\Omega_F(t))}
+\frac{\rho_S}{2} \frac{d}{dt} \|  {\boldsymbol \xi} \|^2_{L^2(\hat{\Omega}_S)}
+\frac{1}{2} \frac{d}{dt}\| {\boldsymbol \eta} \|^2_{S}
\\
 = 
 - \int_{\Gamma_F^{in}} p_{in}(t)  \boldsymbol v \cdot \boldsymbol n  d  S
 - \int_{\Gamma_F^{out}} p_{out}(t)  \boldsymbol v \cdot \boldsymbol n  d  S.
\end{gather*}

%

\section{Numerical method}
Let $\Delta t$ be the time step and $t^n = n \Delta t$ for $n=0, \ldots, N.$ We denote by $z^n$ the approximation of a time-dependent function $z$ at time level $t^n$.  We define the discrete backward difference operator $d_t z^{n+1}$ and the average $z^{n+\frac12}$ as 
$$
d_t z^{n+1} = \frac{z^{n+1}-z^n}{\Delta t},
\qquad z^{n+\frac12} = \frac{z^{n+1}+z^n}{2}.
$$
Similar as in~\cite{badia,badia2009robin}, we consider a linear combination of  FSI coupling conditions~\eqref{kinematic}-\eqref{dynamic}
\begin{align}
& \alpha \boldsymbol \xi+  \boldsymbol{\sigma}_S {\boldsymbol n}_S= \alpha \boldsymbol{v} \circ \mathcal{A}(t)   - {J} \boldsymbol \sigma_F \boldsymbol{F}^{-T} \boldsymbol n_F    \;\;\; \textrm{on} \; \hat\Gamma \times (0,T), \label{lcomb}
 \end{align}
where $\alpha>0$ is a  combination parameter. Using~\eqref{dynamic} again, we introduce the following two time-discrete  transmission conditions of Robin type:
\begin{align}
& \alpha \boldsymbol \xi^{n+1}+  \boldsymbol{\sigma}_S^{n+1} {\boldsymbol n}_S= \alpha \boldsymbol{v}^{n} \circ \mathcal{A}(t^n)   - {J}^n \boldsymbol \sigma_F^n (\boldsymbol{F}^n)^{-T} \boldsymbol n_F^n \;\;\; \textrm{on} \; \hat\Gamma \times (0,T), \label{cc1}
\\
  & 
  \alpha \boldsymbol \xi^{n+1}   - {J}^{n+1} \boldsymbol \sigma_F^{n+1} (\boldsymbol{F}^{n+1})^{-T} \boldsymbol n_F^{n+1}
 = 
\alpha \boldsymbol{v}^{n+1}\circ \mathcal{A}(t^{n+1}) 
 - {J}^n \boldsymbol \sigma_F^n (\boldsymbol{F}^n)^{-T} \boldsymbol n_F^n.
  \quad  \textrm{on} \; \hat\Gamma \times (0,T).
   \label{cc2}
 \end{align}
Condition~\eqref{cc1} will serve as a Robin-type boundary condition for the structure sub-problem, and condition~\eqref{cc2} will serve as a Robin-type boundary condition for the fluid sub-problem. 
To discretize the fluid and structure sub-problems in time, we use the Backward Euler scheme. 
The fluid and structure sub-problems, semi-discretized in time, are now given as follows:

\noindent \textbf{Structure sub-problem:} Find ${\boldsymbol \eta}^{n+1}$ and $\boldsymbol \xi^{n+1}$ such that
\begin{align}
& d_t \boldsymbol \eta^{n+1} = \boldsymbol \xi^{n+1}   &  \textrm{in}\; \hat{\Omega}_S,  
\label{scheme1}
\\
& {\rho}_S d_t  {\boldsymbol \xi}^{n+1}  =  {\nabla} \cdot \boldsymbol \sigma_S (\boldsymbol \eta^{n+1}) &  \textrm{in}\; \hat{\Omega}_S,  \\
& \alpha \boldsymbol \xi^{n+1}+  \boldsymbol{\sigma}_S^{n+1} {\boldsymbol n}_S= \alpha \boldsymbol{v}^{n} \circ \mathcal{A}(t^n)   - {J}^n \boldsymbol \sigma_F^n (\boldsymbol{F}^n)^{-T} \boldsymbol n_F^n  & \textrm{on} \; \hat{\Gamma}. 
\end{align}

\noindent \textbf{Geometry sub-problem:} Find ${\boldsymbol \eta}_F^{n+1}$  such that
\begin{align}
&- \Delta {\boldsymbol \eta}^{n+1}_F = 0 & \textrm{in} \; \hat{\Omega}_F,  \\
&  {\boldsymbol \eta}^{n+1}_F = 0 & \textrm{on} \; \hat{\Gamma}^{in}_F \cup \hat{\Gamma}^{out}_F, \\
&  {\boldsymbol \eta}^{n+1}_F = {\boldsymbol \eta}^{n+1} & \textrm{on} \; \hat{\Gamma},
\end{align}
and ${\boldsymbol w}^{n+1}$ such that
\begin{equation}
{\boldsymbol w}^{n+1} \circ \mathcal{A}(t^{n+1}) = d_t {\boldsymbol \eta}_F^{n+1} \quad \textrm{in} \; \hat{\Omega}_F.
\end{equation}
Compute $\Omega_F(t^{n+1})$ as $\Omega_F(t^{n+1}) = (I + {\boldsymbol \eta}_F^{n+1})(\hat{\Omega}_F)$.
Set $ \boldsymbol v^{n} \circ \mathcal{A}(t^{n})  = \boldsymbol w^{n+1} \circ \mathcal{A}(t^{n+1})$ on $\hat\Gamma$.

\noindent \textbf{Fluid sub-problem:} Find $\boldsymbol v^{n+1}$ and $p^{n+1}$ such that
\begin{align}\label{Tfluid}
& \rho_F \left( J^{n} \frac{ {\boldsymbol{v}}^{n+1} \circ \mathcal{A}(t^{n+1})-{\boldsymbol{v}}^{n} \circ \mathcal{A}(t^{n})}{\Delta t} + J^{n+\frac12} ({\boldsymbol{v}}^{n} \circ \mathcal{A}(t^{n})-\boldsymbol{w}^{n+1} \circ \mathcal{A}(t^{n+1})) \cdot \nabla \boldsymbol{v}^{n+1} \circ \mathcal{A}(t^{n+1}) \right)  
\notag
\\
&\qquad = J^{n+1} \nabla \cdot \boldsymbol\sigma_F({\boldsymbol v}^{n+1}, {p}^{n+1})  \circ \mathcal{A}(t^{n+1}) \qquad  \textrm{in}\; \hat{\Omega}_F, 
\\
&J^{n+1} \nabla \cdot {\boldsymbol{v}}^{n+1} = 0 \qquad  \textrm{in}\; \hat{\Omega}_F, \\
&
\alpha \boldsymbol \xi^{n+1}   - {J}^{n+1} \boldsymbol \sigma_F^{n+1} (\boldsymbol{F}^{n+1})^{-T} \boldsymbol n_F^{n+1}
 = 
\alpha \boldsymbol{v}^{n+1}\circ \mathcal{A}(t^{n+1}) 
 - {J}^n \boldsymbol \sigma_F^n (\boldsymbol{F}^n)^{-T} \boldsymbol n_F^n
  \qquad \textrm{on} \; \hat{\Gamma}. \label{scheme2}
\end{align}
We note that the continuous formulation of the fluid sub-problem is written  on the reference domain due to the use of  different time discretizations of the computational domain for different terms in the equation. However, the  deformed domains, as described in~\eqref{Fweak}, are considered in practice.

\subsection{Weak formulation of the semi-discrete partitioned scheme}

We define the following bilinear forms associated with the fluid problem:
\begin{gather*}
 a_F^n(\boldsymbol v, \boldsymbol \phi) = 2 \mu_F \int_{\Omega_F(t^n)} \boldsymbol
D(\boldsymbol v) : \boldsymbol D (\boldsymbol \phi) d \boldsymbol x, 
\qquad  b_F^n(p, \boldsymbol \phi) =  \int_{\Omega_F (t^n)} p \nabla \cdot \boldsymbol
\phi d \boldsymbol x,
\end{gather*}
for all $\boldsymbol v, \boldsymbol \phi \in V^F(t^n)$ and $p \in Q^F(t^n)$. To simplify the notation  moving forward, we will write
$$
\int_{\Omega(t^{m})} \boldsymbol v^{n}
\qquad 
\textrm{instead of}
\quad
\int_{\Omega(t^{m})} \boldsymbol v^{n} \circ \mathcal{A}(t^{n}) \circ \mathcal{A}^{-1}(t^{m})
$$
whenever we need to integrate $\boldsymbol v^n$ on a domain $\Omega(t^{m})$, for $m \neq n$.
The weak formulation of the fluid and structure sub-problems is given as:

\noindent \textbf{Structure sub-problem:} Find $\boldsymbol \xi^{n+1} \in V^S$ and $\boldsymbol \eta^{n+1} \in V^S$, where $\boldsymbol \xi^{n+1} = d_t \boldsymbol \eta^{n+1}$, such that for all $\boldsymbol \zeta \in V^S$ we have
\begin{align}
\rho_S \int_{\hat{\Omega}_S} d_t \boldsymbol \xi^{n+1} \cdot \boldsymbol \zeta d \boldsymbol x + a_S(\boldsymbol \eta^{n+1}, \boldsymbol \zeta)
+\alpha \int_{\hat{\Gamma}} (\boldsymbol \xi^{n+1} - \boldsymbol v^n ) \cdot \boldsymbol \zeta d \boldsymbol x
= -\int_{\hat\Gamma} {J}^n \boldsymbol \sigma_F^n (\boldsymbol{F}^n)^{-T} \boldsymbol n_F^n \cdot \boldsymbol \zeta d\boldsymbol x. \label{Sweak} 
\end{align}

\noindent \textbf{Fluid sub-problem:} Find $\boldsymbol v^{n+1} \in V^F(t^{n+1})$ and $p^{n+1} \in Q^F(t^{n+1})$ such that for all $\boldsymbol \phi \in V^F(t^{n+1})$ and $\psi \in Q^F(t^{n+1})$ we have
\begin{align}
&\rho_F \int_{\Omega_F(t^{n})}  \frac{ {\boldsymbol{v}}^{n+1}-{\boldsymbol{v}}^{n}}{\Delta t}  \cdot {\boldsymbol \phi} d \boldsymbol x 
+\rho_F \int_{\Omega_F(t^{n+\frac12})}   \left(({\boldsymbol{v}}^{n}-{\boldsymbol{w}}^{n+1}) \cdot \nabla \right) {\boldsymbol{v}}^{n+1}  \cdot {\boldsymbol \phi} d \boldsymbol x 
+a_F^{n+1}({\boldsymbol v}^{n+1}, \boldsymbol \phi)
\notag
\\
&
\quad
-b_F^{n+1}({p}^{n+1}, \boldsymbol \phi)
+b_F^{n+1}(\psi, {\boldsymbol v}^{n+1})
+ \alpha  \int_{\Gamma(t^{n+1})}(\boldsymbol v^{n+1} - \boldsymbol \xi^{n+1} ) \cdot \boldsymbol \phi d \boldsymbol x
\notag
\\
&
\quad
 =\int_{\Gamma(t^{n})}  \boldsymbol \sigma_F(\boldsymbol v^n, p^n) \boldsymbol n_F^n \cdot {\boldsymbol \phi} d \boldsymbol x
 +\int_{\Gamma_F^{in} \cup \Gamma_F^{out}} \boldsymbol \sigma_F (\boldsymbol v^{n+1}, p^{n+1}) \boldsymbol n_F^{n+1} \cdot \boldsymbol \phi d \boldsymbol x.
\label{Fweak}
 \end{align}
 We note that the boundary conditions in the fluid sub-problem are not specified. Conditions~\eqref{inlet}-\eqref{outlet} will be used in numerical simulations in Section~\ref{numerics}, while conditions~\eqref{inlet2}-\eqref{dp3} will be used in stability analysis in Section~\ref{StabAn}.

\section{Stability analysis}~\label{StabAn}
Let $\mathcal{E}^n$ denote the sum of the kinetic energy of the fluid and the kinetic and elastic energy of the solid, given by
$$
\mathcal{E}^n= \frac{\rho_F}{2} \| \boldsymbol v^{n} \|^2_{L^2(\Omega_F(t^n))} 
+ \frac{\rho_S}{2}  \| \boldsymbol \xi^{n} \|^2_{L^2(\hat{\Omega}_S)}
+\frac12   \| \boldsymbol \eta^{n}\|^2_S,
$$ 
let $\mathcal{D}^n$ denote the fluid viscous dissipation, given by
$$
\mathcal{D}^n = \mu_F \Delta t \sum_{k=1}^{n}   \| \boldsymbol D(\boldsymbol v^{k}) \|^2_{L^2(\Omega_F(t^{k}))},
$$
and let $\mathcal{N}_1^n$ and $\mathcal{N}_2^n$ denote terms due to numerical dissipation, given by
\begin{align*}
& \mathcal{N}_1^n = 
\frac{\alpha \Delta t}{2}  \| \boldsymbol v^{n}\|^2_{L^2(\hat{\Gamma})} 
+\frac{ \Delta t}{2 \alpha}   \| {J}^n \boldsymbol \sigma_F^n (\boldsymbol{F}^n)^{-T} \boldsymbol n_F^n  \|^2_{L^2(\hat\Gamma)} , 
 \\
&  \mathcal{N}_2^n =
  \frac{\rho_S}{2} \sum_{k=0}^{n-1} \| \boldsymbol \xi^{k+1} - \boldsymbol \xi^{k} \|^2_{L^2(\hat{\Omega}_S)}  
+\frac12 \sum_{k=0}^{n-1} \| \boldsymbol \eta^{k+1} - \boldsymbol \eta^k\|^2_S 
 + \frac{\rho_F}{2} \sum_{k=0}^{n-1}  \| {\boldsymbol v}^{k+1} - \boldsymbol v^{k} \|^2_{L^2(\Omega_F(t^{k}))} 
  \\
 &\qquad 
+\frac{\alpha \Delta t}{2} \sum_{k=0}^{n-1}  \| \boldsymbol \xi^{k+1}-\boldsymbol v^k \|^2_{L^2(\hat{\Gamma})}. 
\end{align*}
The stability of method~\eqref{Sweak}-\eqref{Fweak} is presented in the following theorem. 
\begin{theorem}
Let $(\boldsymbol \xi^n, \boldsymbol \eta^n, \boldsymbol v^n, p^n$) be the solution of~\eqref{Sweak}-\eqref{Fweak}. Assume  boundary conditions~\eqref{inlet2}-\eqref{dp3} are imposed. 
Then, the following a priori energy estimate holds:
\begin{align}
& \mathcal{E}^{N}+\mathcal{D}^N+\mathcal{N}_1^N+\mathcal{N}_2^N
\leq  \mathcal{E}^0+\mathcal{N}_1^0 + \frac{\Delta t C_P^2 C_K^2 }{2 \mu_F} \| p_{in} \|^2_{L^2(\Gamma_F^{in})}
  +  \frac{\Delta t C_P^2 C_K^2}{2 \mu_F} \| p_{out} \|^2_{L^2(\Gamma_F^{in})}.
    \label{energy_inequality} 
\end{align}

\end{theorem}

\begin{proof}
Take  $\boldsymbol \zeta =\Delta t \boldsymbol \xi^{n+1}$ in~\eqref{Sweak} and $\boldsymbol \phi = \Delta t\boldsymbol v^{n+1}, \psi =\Delta t p^{n+1}$ in~\eqref{Fweak}. Adding the equations and recasting the interface integrals in the fluid problem on the reference domain, we have
\begin{align}
& \frac{\rho_S}{2} \left( \| \boldsymbol \xi^{n+1} \|^2_{L^2(\hat{\Omega}_S)} - \| \boldsymbol \xi^{n} \|^2_{L^2(\hat{\Omega}_S)} + \| \boldsymbol \xi^{n+1} - \boldsymbol \xi^{n} \|^2_{L^2(\hat{\Omega}_S)} \right)
+\frac12 \left( \| \boldsymbol \eta^{n+1}\|^2_S - \| \boldsymbol \eta^{n}\|^2_S + \| \boldsymbol \eta^{n+1} - \boldsymbol \eta^n\|^2_S \right)
\notag \\
&
+\rho_F \int_{\Omega_F(t^{n})}  ({ \boldsymbol{v}}^{n+1}-{\boldsymbol{v}}^{n}) \cdot {\boldsymbol v}^{n+1} d \boldsymbol x
+\rho_F \Delta t \int_{\Omega_F( t^{n+\frac12})}   \left(   ({\boldsymbol{v}}^{n}-{\boldsymbol{w}}^{n+1} ) \cdot \nabla \right) {\boldsymbol{v}}^{n+1}  \cdot {\boldsymbol v}^{n+1} d \boldsymbol x 
\notag 
\\
&
+2\mu_F \Delta t \| \boldsymbol D(\boldsymbol v^{n+1}) \|^2_{L^2(\Omega_F(t^{n+1}))}
+\frac{\alpha \Delta t}{2} \left( \| \boldsymbol v^{n+1}  \|^2_{L^2(\hat\Gamma)} 
-\| \boldsymbol v^n   \|^2_{L^2(\hat\Gamma)} 
\right)
\notag 
\\
&
+\frac{\alpha \Delta t}{2} \left( 
\| \boldsymbol \xi^{n+1}-\boldsymbol v^n   \|^2_{L^2(\hat\Gamma)}
+\| \boldsymbol v^{n+1}  - \boldsymbol \xi^{n+1} \|^2_{L^2(\hat \Gamma)} 
\right)
\notag
\\
&
= 
 \Delta t \int_{\hat\Gamma} {J}^n \boldsymbol \sigma_F^n (\boldsymbol{F}^n)^{-T} \boldsymbol n_F^n \cdot (\boldsymbol v^{n+1} - \boldsymbol \xi^{n+1}) d \boldsymbol x
  -\Delta t \int_{\Gamma_F^{in}} p_{in} \boldsymbol v^{n+1} \cdot \boldsymbol n_F^{n+1} d x
 \notag
\\
&
  -\Delta t \int_{\Gamma_F^{out}} p_{out} \boldsymbol v^{n+1} \cdot \boldsymbol n_F^{n+1} dx
  +\frac{\rho_F \Delta t}{2} \int_{\Gamma_F^{in} \cup \Gamma_F^{out}} |\boldsymbol v^{n+1} |^2 \boldsymbol v^{n+1} \cdot \boldsymbol{n}_F^{n+1} dx. 
 \label{stab1}
\end{align}
We transform the integral containing the time-derivative of the fluid velocity to the reference domain as follows:
\begin{gather*}
\rho_F \int_{\Omega_F(t^{n})}  ( {\boldsymbol{v}}^{n+1}-{\boldsymbol{v}}^{n})  \cdot {\boldsymbol v}^{n+1} d \boldsymbol x 
=\rho_F \int_{\hat{\Omega}_F}  J^{n} \left( \boldsymbol{v}^{n+1}-{\boldsymbol{v}}^{n} \right)  \cdot \boldsymbol v^{n+1}  d \boldsymbol x.
\end{gather*}
Using identity 
\begin{align*}
&\int_{\hat{\Omega}_F}  J^{n} \left( \boldsymbol{v}^{n+1}-{\boldsymbol{v}}^{n} \right)  \cdot \boldsymbol v^{n+1}  d \boldsymbol x
\notag
\\
& \quad 
=
\frac12 \int_{\hat{\Omega}_F} \left(J^{n+1} |\boldsymbol{v}^{n+1}|^2 - J^n |\boldsymbol{v}^{n} |^2 \right) d \boldsymbol x
-\frac12 \int_{\hat{\Omega}_F} \left(J^{n+1}-J^n \right) |\boldsymbol{v}^{n+1} |^2 d \boldsymbol x
\notag
\\
& \quad 
+\frac{1}{2 } \int_{\hat{\Omega}_F} J^n |\boldsymbol{v}^{n+1} -\boldsymbol{v}^{n} |^2 d \boldsymbol x,
\end{align*}
we obtain
\begin{align}
\rho_F \int_{\Omega_F(t^{n})}  \left( {\boldsymbol{v}}^{n+1}-{\boldsymbol{v}}^{n} \right)  \cdot {\boldsymbol v}^{n+1} d \boldsymbol x 
=&\frac{\rho_F}{2 } \left( \|\boldsymbol v^{n+1} \|^2_{L^2(\Omega_F(t^{n+1}))}  -\|\boldsymbol v^{n} \|^2_{L^2(\Omega_F(t^{n}))}+\|{\boldsymbol v}^{n+1}-\boldsymbol v^{n} \|^2_{L^2(\Omega_F(t^{n}))}   \right)
\notag
\\
&-\frac{\rho_F}{2} \int_{\hat{\Omega}_F} \left( J^{n+1}-J^n \right) |\boldsymbol{v}^{n+1}|^2 d \boldsymbol x. \label{kineticTerm}
\end{align}
To handle the last term in~\eqref{kineticTerm}, we use the~\emph{Geometric Conservation Law}~\cite{boffi2004stability,nobile1999stability,lukavcova2013kinematic,donea2004arbitrary} given as
\begin{align*}
\| \boldsymbol v^{n+1} \|^2_{L^2(\Omega_F(t^{n+1}))}
-\| \boldsymbol v^{n+1} \|^2_{L^2(\Omega_F(t^{n}))}
= \int_{t^n}^{t^{n+1}}\left(  \int_{\Omega_F(t)}|\boldsymbol v^{n+1}|^2 \nabla \cdot \boldsymbol w d \boldsymbol x \right) dt.
\end{align*}
Since we consider a linear time variation for the displacement
of the points of the fluid domain, the domain velocity is constant in time interval $[t^n, t^{n+1}]$. In that case, it has been shown in~\cite{lesoinne1996geometric} that the Geometric Conservation Law is exactly satisfied if the midpoint formula is used for time-integration in two-dimensions,  yielding
\begin{align}
\| \boldsymbol v^{n+1} \|^2_{L^2(\Omega_F(t^{n+1}))}
-\| \boldsymbol v^{n+1} \|^2_{L^2(\Omega_F(t^{n}))}
= \Delta t \int_{\Omega_F(t^{n+\frac12})}|\boldsymbol v^{n+1}|^2 \nabla \cdot \boldsymbol w^{n+\frac12} d \boldsymbol x.
\label{GCL}
\end{align}
As in~\cite{nobile2008effective}, we note that since the domain velocity is piecewise constant, we have $\boldsymbol w^{n+\frac12}=\boldsymbol w^{n+1}$. Therefore, equation~\eqref{kineticTerm} can be written as
\begin{align}
\rho_F \int_{\Omega_F(t^{n})}  \left( {\boldsymbol{v}}^{n+1}-{\boldsymbol{v}}^{n} \right)  \cdot {\boldsymbol v}^{n+1} d \boldsymbol x 
=&\frac{\rho_F}{2 } \left( \|\boldsymbol v^{n+1} \|^2_{L^2(\Omega_F(t^{n+1}))}  -\|\boldsymbol v^{n} \|^2_{L^2(\Omega_F(t^{n}))}+\|{\boldsymbol v}^{n+1}-\boldsymbol v^{n} \|^2_{L^2(\Omega_F(t^{n}))}   \right)
\notag
\\
&-\frac{\rho_F \Delta t}{2}\int_{\Omega_F(t^{n+\frac12})}|\boldsymbol v^{n+1}|^2 \nabla \cdot \boldsymbol w^{n+1} d \boldsymbol x.
\label{timeterm}
\end{align}
For the advection term, we proceed as follows:
\begin{align}
\rho_F \Delta t \int_{\Omega_F(t^{n+\frac12})}   \left(({\boldsymbol{v}}^{n}-{\boldsymbol{w}}^{n+1}) \cdot \nabla \right)  {\boldsymbol{v}}^{n+1} \cdot  {\boldsymbol{v}}^{n+1}  d \boldsymbol x
=
 \frac{\rho_F \Delta t}{2}  \int_{\Omega_F(t^{n+\frac12})} \nabla \cdot {\boldsymbol w}^{n+1}  |{\boldsymbol v}^{n+1} |^2 d\boldsymbol x 
 \notag
 \\
 + \frac{\rho_F \Delta t}{2} \int_{\Gamma_F^{in} \cup \Gamma_F^{out}} |\boldsymbol v^{n+1} |^2 \boldsymbol v^{n+1} \cdot \boldsymbol{n}_F^{n+1} dx.  
  \label{advection_simplified} 
\end{align}
To handle the interface term in~\eqref{stab1}, using~\eqref{scheme2}, we have
\begin{align}
& \Delta t \int_{\hat\Gamma} {J}^n \boldsymbol \sigma_F^n (\boldsymbol{F}^n)^{-T} \boldsymbol n_F^n  \cdot (\boldsymbol v^{n+1} - \boldsymbol \xi^{n+1}) d \boldsymbol x
\notag 
\\
&
=\frac{ \Delta t}{\alpha} \int_{\hat\Gamma} {J}^n \boldsymbol \sigma_F^n (\boldsymbol{F}^n)^{-T} \boldsymbol n_F^n  \cdot \left(   {J}^n \boldsymbol \sigma_F^n (\boldsymbol{F}^n)^{-T} \boldsymbol n_F^n -{J}^{n+1} \boldsymbol \sigma_F^{n+1} (\boldsymbol{F}^{n+1})^{-T} \boldsymbol n_F^{n+1} \right) 
\notag 
\\
&
=
\frac{ \Delta t}{2 \alpha} \left( \| {J}^n \boldsymbol \sigma_F^n (\boldsymbol{F}^n)^{-T} \boldsymbol n_F^n  \|^2_{L^2(\hat\Gamma)} 
- \| {J}^{n+1} \boldsymbol \sigma_F^{n+1} (\boldsymbol{F}^{n+1})^{-T} \boldsymbol n_F^{n+1} \|^2_{L^2(\hat\Gamma)} \right)
\notag 
\\
&
\qquad
 +\frac{ \Delta t}{2 \alpha} \| {J}^n \boldsymbol \sigma_F^n (\boldsymbol{F}^n)^{-T} \boldsymbol n_F^n-{J}^{n+1} \boldsymbol \sigma_F^{n+1} (\boldsymbol{F}^{n+1})^{-T} \boldsymbol n_F^{n+1} \|^2_{L^2(\hat\Gamma) }
\notag
\\
&=
\frac{ \Delta t}{2 \alpha} \left( \| {J}^n \boldsymbol \sigma_F^n (\boldsymbol{F}^n)^{-T} \boldsymbol n_F^n  \|^2_{L^2(\hat\Gamma)} 
- \| {J}^{n+1} \boldsymbol \sigma_F^{n+1} (\boldsymbol{F}^{n+1})^{-T} \boldsymbol n_F^{n+1} \|^2_{L^2(\hat\Gamma)} \right)
\notag
\\
&
\qquad  +\frac{ \alpha \Delta t }{2 }  \| \boldsymbol v^{n+1}  - \boldsymbol \xi^{n+1} \|^2_{L^2(\hat\Gamma)}.
  \label{energyB}
\end{align}
To estimate the forcing terms, we use the Cauchy-Schwarz, Young's, Poincare and Korn's inequalities as follows:
\begin{align}
 &
 -\Delta t \int_{\Gamma_F^{in}} p_{in} \boldsymbol v_h^{n+1} \cdot \boldsymbol n_F
  -\Delta t \int_{\Gamma_F^{out}} p_{out} \boldsymbol v_h^{n+1} \cdot \boldsymbol n_F
  \notag
  \\
 &
  \leq 
  \frac{\Delta t C_P^2 C_K^2 }{2 \mu_F} \| p_{in} \|^2_{L^2(\Gamma_F^{in})}
  +  \frac{\Delta t C_P^2 C_K^2}{2 \mu_F} \| p_{out} \|^2_{L^2(\Gamma_F^{in})}
  +\Delta t \mu_F \| \boldsymbol D( \boldsymbol v^{n+1}) \|^2_{L^2(\Omega_F(t^{n+1}))}.
    \label{inequality} 
\end{align}
Using~\eqref{timeterm}-\eqref{inequality} in~\eqref{stab1} and summing from $n=0$ to $N-1$ completes the proof.

\end{proof}

\begin{remark}
Similarly as in~\cite{nobile2008effective,badia2009robin,badia2008fluid,gerardo2010analysis}, the method proposed here is developed using generalized Robin boundary conditions. However, in this work, generalized Robin boundary conditions are designed and discretized in a novel way, leading to an unconditionally stable scheme which does not require sub-iterations. As opposed to the previous work, where two combination parameters are introduced, we have only one combination parameter, $\alpha$. 

This method also exhibits similarities to the method proposed in~\cite{burman2009stabilization}. In particular, the weak form of the partitioned scheme presented in this work is similar to the incomplete version of the explicit method presented in~\cite{burman2009stabilization}, which was obtained by enforcing coupling conditions using Nitsche's penalty method.  However, only conditional stability was proved for the method presented in~\cite{burman2009stabilization} after a stabilization term was added.   
\end{remark}

\section{Convergence analysis}\label{conv}
To analyze the convergence of the fully discrete proposed method, we  assume that the fluid is described by the time-dependent Stokes equations, that the structure deformation is infinitesimal and that the fluid-structure interaction is linear. 
 These assumptions are common  in the analysis of partitioned schemes for FSI problems as the main difficulties related to the splitting between the fluid and structure sub-problems are still present~\cite{fernandez2015generalized,bukavc2016stability,burman2009stabilization,banks2014analysis2}. Therefore, to simplify the notation, in the following we will omit the hat notation. 
The resulting numerical method is given by:

\noindent \textbf{Structure sub-problem:} Find ${\boldsymbol \eta}^{n+1}$ and $\boldsymbol \xi^{n+1} = d_t {\boldsymbol \eta}^{n+1}$ such that
\begin{align}
& {\rho}_S d_t  {\boldsymbol \xi}^{n+1}  =  {\nabla} \cdot \boldsymbol \sigma_S (\boldsymbol \eta^{n+1}) &  \textrm{in}\; {\Omega}_S,  
\label{Ssolid}\\
& \alpha {\boldsymbol \xi}^{n+1} +  \boldsymbol \sigma_S ({\boldsymbol \eta}^{n+1}) {\boldsymbol n}_S=\alpha{\boldsymbol{v}}^{n}  -  \boldsymbol \sigma_F (\boldsymbol v^{n}, p^{n})  \boldsymbol n_F & \textrm{on} \; {\Gamma}. 
\end{align}

\noindent \textbf{Fluid sub-problem:}  Find $\boldsymbol v^{n+1}$ and $p^{n+1}$ such that
\begin{align}
& \rho_F  d_t  \boldsymbol{v}^{n+1}  = \nabla \cdot \boldsymbol\sigma_F({\boldsymbol v}^{n+1}, {p}^{n+1})  &  \textrm{in}\; {\Omega}_F, 
\\
& \nabla \cdot {\boldsymbol{v}}^{n+1} = 0 &  \textrm{in}\; {\Omega}_F, \\
&  \alpha \boldsymbol{v}^{n+1}
+ \boldsymbol \sigma_F (\boldsymbol v^{n+1}, p^{n+1}) \boldsymbol n_F
= 
 \alpha \boldsymbol \xi^{n+1} 
 + {\boldsymbol \sigma}_F(\boldsymbol v^{n}, p^{n}) {\boldsymbol n}_F
  & \textrm{on} \; {\Gamma}.
  \label{Sfluid}
\end{align}

To discretize~\eqref{Ssolid}-\eqref{Sfluid} in space, we use the finite element method. The finite element spaces are  defined as the subspaces $V^F_h \subset V^F, Q^F_h \subset Q^F$ and $V^S_h \subset V^S$ based on a conforming finite element triangulation with maximum triangle diameter $h$. We  assume that spaces $V^F_h$ and $Q^F_h$ are \textit{inf-sup} stable and that the fluid boundary conditions are~\eqref{inlet}-\eqref{outlet}.
The weak formulation of the scheme is given as follows: 

\noindent \textbf{Structure sub-problem:} Find $\boldsymbol \xi_h^{n+1} \in V^S_h$ and $\boldsymbol \eta_h^{n+1} \in V^S_h$, where $\boldsymbol \xi^{n+1}_h = d_t \boldsymbol \eta_h^{n+1}$, such that for all $\boldsymbol \zeta_h \in V^S_h$ we have
\begin{align}
\rho_S \int_{{\Omega}_S} d_t \boldsymbol \xi_h^{n+1} \cdot \boldsymbol \zeta_h d \boldsymbol x + a_S(\boldsymbol \eta_h^{n+1}, \boldsymbol \zeta_h)
+\alpha \int_{{\Gamma}} (\boldsymbol \xi_h^{n+1} - \boldsymbol v_h^n  ) \cdot \boldsymbol \zeta_h d \boldsymbol x
= -\int_{\Gamma} \boldsymbol \sigma_F (\boldsymbol v_h^{n}, p_h^{n})  \boldsymbol n_F \cdot \boldsymbol \zeta_h d\boldsymbol x. \label{SSweak}
\end{align}

\noindent \textbf{Fluid sub-problem:} Find $\boldsymbol v_h^{n+1} \in V^F_h$ and $p_h^{n+1} \in Q^F_h$ such that for all $\boldsymbol \phi_h \in V^F_h$ and $\psi_h \in Q^F_h$ we have
\begin{align}
&\rho_F \int_{\Omega_F}  d_t \boldsymbol{v}_h^{n+1}  \cdot {\boldsymbol \phi_h} d \boldsymbol x 
+a_F({\boldsymbol v}_h^{n+1}, \boldsymbol \phi_h)
-b_F({p}^{n+1}_h, \boldsymbol \phi_h)
+b_F(\psi_h, {\boldsymbol v}^{n+1}_h)
+ \alpha  \int_{\Gamma}(\boldsymbol v_h^{n+1} - \boldsymbol \xi_h^{n+1} ) \cdot \boldsymbol \phi_h d \boldsymbol x
\notag
\\
&
\quad
 =\int_{\Gamma} {\boldsymbol \sigma}_F({\boldsymbol v}_h^{n}, {p}_h^{n}) {\boldsymbol n}_F  \cdot \boldsymbol \phi_h d \boldsymbol x
 -\int_{\Gamma_F^{in}  } p_{in}(t)  \boldsymbol \phi_h \cdot \boldsymbol n_F dx
 -\int_{  \Gamma_F^{out}} p_{out}(t) \boldsymbol \phi_h \cdot \boldsymbol n_F  dx.
\label{SFweak}
 \end{align}

For spatial discretization, we  use the Lagrangian finite elements
of polynomial degree $k$ for all  variables except for the fluid pressure
for which we use elements of degree $r<k$.  
Assume that the continuous solution satisfies the following assumptions:
\begin{align}
&\boldsymbol v \in  L^{\infty}(0,T; H^{k+1}(\Omega_F)) \cap H^1(0,T; H^{k+1}(\Omega_F))\cap H^2(0,T; L^2(\Omega_F)), \label{reg1} \\
&\boldsymbol v|_{\Gamma} \in L^{\infty} (0,T; H^{k+1}(\Gamma)) \cap H^1(0,T; H^{k+1}(\Gamma)) , \\
& p  \in L^2(0,T; H^{r+1}(\Omega_F)),\;\;  p|_{\Gamma} \in H^{1}(0,T; L^2(\Gamma)), \\
&\boldsymbol \eta  \in W^{1,\infty}(0,T; H^{k+1}(\Omega_S)) \cap H^2(0,T; H^{k+1}(\Omega_S))\cap H^3(0,T; L^2(\Omega_S)). \label{reg2}
\end{align}

Let $a \ltt(\gtt) b$ denote that there exists a positive constant $C$, independent of $h$ and $\Delta t$, such that $a \leq(\geq) C b$.
 We introduce the following time discrete norms:
 \begin{equation*}
\|\boldsymbol \varphi\|_{L^2(0,T; X)} = \left(\Delta t \sum_{n=0}^{N-1}
\|\boldsymbol \varphi^{n+1}\|^2_{X} \right)^{\frac12}, \quad
\| \boldsymbol \varphi\|_{L^{\infty}(0,T; X)} = \max_{0 \le n \le N} \|\boldsymbol
\varphi^n \|_{X}, \label{tdiscnorm}
 \end{equation*}
where $X \in \{H^k(\Omega_F), H^k(\Omega_S), H^k(\Gamma),  S\}$.  Note that they are equivalent
to the continuous norms since we use piecewise constant approximations in
time. 
Furthermore, the following inequality holds:
\begin{equation*}
\Delta{t}\sum_{n=1}^{N-1} \| d_t\boldsymbol \varphi^{n+1} \|^2_{X}
\ltt \| \partial_t\boldsymbol \varphi \|^2_{L^2(0,T; X)}.
\label{ineq}
\end{equation*}

Let $P_h$ be the Lagrangian interpolation operator onto $V^S_h.$ Then, $I_h := P_h|_{\Gamma}$ is a Lagrangian interpolation operator.  Similar as in~\cite{bukavc2016stability,Fernandez2012incremental}, we introduce a Stokes-like projection
operator $(S_h, R_h): V^F \rightarrow V^F_h \times Q_h^F$, defined for all
$\boldsymbol v \in V^F$ by
\begin{align}
& (S_h \boldsymbol v, R_h \boldsymbol v) \in V^F_h \times Q^F_h,  \\
& (S_h \boldsymbol v)|_{\Gamma} = I_h (\boldsymbol v|_{\Gamma}),   \\
& a_F (S_h \boldsymbol v, \boldsymbol \varphi_h)
-b_F (R_h \boldsymbol v, \boldsymbol \varphi_h) 
= a_F (\boldsymbol v, \boldsymbol \varphi_h), \; \forall \boldsymbol \varphi_h \in V^F_h \; \textrm{such that} \; \boldsymbol \varphi_h|_{\Gamma}=0, \\
&b_F( q, S_h \boldsymbol v) = 0, \quad \forall q \in Q^F_h.   \label{press_proj}
\end{align} 
Projection operators $S_h$  and $I_h$ satisfy the following approximation properties
(see~\cite{ciarlet1978finite,bukavc2012fluid}):  
\begin{align}
& \| \boldsymbol D( \boldsymbol v - S_h \boldsymbol v) \|_{L^2(\Omega_F)} \ltt h^k \| \boldsymbol
v \|_{H^{k+1}(\Omega_F)} \quad \textrm{for all} \; \boldsymbol v \in V^F,
  \label{app1}
\\
&
\| \boldsymbol \xi - I_h \boldsymbol \xi \|_{L^2(\Gamma)}
 + h \| \boldsymbol \xi - I_h \boldsymbol \xi \|_{H^1(\Gamma)} \ltt h^{k+1} \| \boldsymbol \xi
\|_{H^{k+1}(\Gamma)}  \quad \textrm{for all} \; \boldsymbol \xi \in V^S. 
\end{align}
Let $\Pi_h$ be a projection operator onto $Q_h^F$ such that
\begin{equation}
\| p - \Pi_h p \|_{L^2(\Omega_F)} \ltt  h^{r+1} \| p \|_{H^{r+1}(\Omega_F)},
\quad \textrm{for all} \;  p \in Q^F.
\end{equation}
Let 
$R_h$ be the  Ritz projector onto  $V_h^S$ such that for all $\boldsymbol
\eta \in V^S$,
\begin{equation}
a_S(\boldsymbol \eta - R_h \boldsymbol \eta, \boldsymbol \chi_h) = 0  \quad
\textrm{for all} \; \boldsymbol \chi_h \in V_h^S. \label{Ritz}
\end{equation}
Then, the finite element theory for Ritz projections~\cite{ciarlet1978finite}
gives
\begin{equation}
\| \boldsymbol \eta - R_h \boldsymbol \eta \|_{S}  \ltt  h^{k} \|
\boldsymbol \eta \|_{H^{k+1}(\Gamma)} \quad \textrm{for all} \; \boldsymbol \eta \in
V^S.
 \label{app2}
\end{equation}
In the following, in addition to standard inequalities~\cite{bukavc2012fluid}, we will also use the discrete trace-inverse inequality: For a triangular domain $\Omega_F \subset \mathbb{R}^2$ there exists a positive
constant $C_{TI}$ depending on the angles in the finite element mesh such that
\begin{gather}
\| \boldsymbol v_h \|^2_{L^2(\Gamma)} \leq \frac{C_{TI} k^2}{h} \| \boldsymbol v_h \|^2_{L^2(\Omega_F)},  \label{traceinverse}
\end{gather}
for all $\boldsymbol v_h \in V_h.$

We assume that the continuous fluid velocity belongs to the space $V^{FD}=\{\boldsymbol v \in V^F | \; \nabla \cdot \boldsymbol v =0\}$. Since the test functions for the partitioned scheme do not satisfy the  kinematic coupling condition, we start by deriving the monolithic variational formulation with the test functions in $V_h^S \times V_h^{F} \times Q^F_h$: Find $(\boldsymbol \xi^{n+1} =\partial_t \boldsymbol \eta^{n+1}, \boldsymbol v^{n+1}, p^{n+1}) \in  V^{S}  \times V^F \times Q^F$ with $\boldsymbol v^{n+1} = \boldsymbol \xi^{n+1}$ on $\Gamma$ such that for all $(\boldsymbol \zeta_h, \boldsymbol \phi_h) \in V_h^{S}\times V_h^F$ we have
\begin{align}
& \rho_F \int_{\Omega_F} \partial_t \boldsymbol v^{n+1} \cdot \boldsymbol \phi_h 
+a_F(\boldsymbol v^{n+1}, \boldsymbol \phi_h)
-b_F(p^{n+1}, \boldsymbol \phi_h)
+ \rho_S \int_{\Omega_S} \partial_t \boldsymbol \xi_h^{n+1} \cdot  \boldsymbol \zeta_h 
+a_S(\boldsymbol \eta, \boldsymbol \zeta_h)
\notag \\
 &= \int_{\Gamma} \boldsymbol \sigma_F (\boldsymbol v^{n+1},p^{n+1}) \boldsymbol n_F \cdot (\boldsymbol \phi_h - \boldsymbol \zeta_h) 
- \int_{\Gamma_{in}} p_{in}(t^{n+1}) \boldsymbol \phi_h \cdot \boldsymbol n
- \int_{\Gamma_{out}} p_{out}(t^{n+1}) \boldsymbol \phi_h \cdot \boldsymbol n.
\label{monoweak}
\end{align}
Subtracting~\eqref{SSweak}-\eqref{SFweak} from~\eqref{monoweak}, we obtain the following error equation:
\begin{align}
& \rho_F \int_{\Omega_F} d_t( \boldsymbol v^{n+1}- \boldsymbol v_h^{n+1})  \cdot \boldsymbol \phi_h 
+a_F(\boldsymbol v^{n+1} - \boldsymbol v_h^{n+1}, \boldsymbol \phi_h)
-b_F (p^{n+1}-p_h^{n+1}, \boldsymbol \phi_h)
-b_F( \psi_h, \boldsymbol v_h^{n+1})
\notag \\
& \;
+ \rho_S \int_{\Omega_S}  d_t( \boldsymbol \xi^{n+1}- \boldsymbol \xi_h^{n+1})  \cdot \boldsymbol \zeta_h  
+a_S(\boldsymbol \eta^{n+1} - \boldsymbol \eta_h^{n+1}, \boldsymbol \zeta_h)
+\alpha \int_{\Gamma} (\boldsymbol \xi^{n+1} - \boldsymbol \xi_h^{n+1} - \boldsymbol v^{n}+\boldsymbol v_h^n) \cdot \boldsymbol \zeta_h d \boldsymbol x
\notag \\
& \;
+\alpha  \int_{\Gamma}(\boldsymbol v^{n+1} - \boldsymbol v_h^{n+1}-\boldsymbol \xi^{n+1} + \boldsymbol \xi_h^{n+1}) \cdot \boldsymbol \phi_h d \boldsymbol x
\notag \\
&
 = \int_{\Gamma} \boldsymbol \sigma_F (\boldsymbol v^{n} - \boldsymbol v_h^n,p^{n}-p_h^n) \boldsymbol n_F \cdot (\boldsymbol \phi_h - \boldsymbol \zeta_h) 
+\mathcal{R}_1 (\boldsymbol \phi_h, \boldsymbol \zeta_h),
 \label{erroreq}
\end{align}
 for all $(\boldsymbol \zeta_h, \boldsymbol \phi_h, \psi_h) \in V_h^{S}\times V_h^F \times Q_h^F$, where, since $\boldsymbol v^{n+1}=\boldsymbol \xi^{n+1}$ on  $\Gamma,$  
\begin{align*}
\mathcal{R}_1 (\boldsymbol \phi_h, \boldsymbol \zeta_h)  &= \rho_F \int_{\Omega_F} (d_t \boldsymbol v^{n+1} - \partial_t \boldsymbol v^{n+1}) \cdot \boldsymbol \phi_h  
+\rho_S \int_{\Omega_S} (d_t \boldsymbol \xi^{n+1} - \partial_t \boldsymbol \xi^{n+1}) \cdot \boldsymbol \zeta_h  
\notag \\
& \; +\alpha \int_{\Gamma} (\boldsymbol v^{n+1}  - \boldsymbol v^{n}) \cdot \boldsymbol \zeta_h d \boldsymbol x
+\int_{\Gamma} \boldsymbol \sigma_F (\boldsymbol v^{n+1} - \boldsymbol v^n,p^{n+1}-p^n) \boldsymbol n_F \cdot  (\boldsymbol \phi_h - \boldsymbol \zeta_h) .
\end{align*}

We split the error of the method as a sum of the approximation error, $\theta_r^{n+1}$,
and the truncation error, $\delta_r^{n+1},$ for $r \in \{F,P,\eta, \xi \}$
as follows:
\begin{align}
\boldsymbol e_F^{n+1} & =\boldsymbol v^{n+1}-\boldsymbol v_h^{n+1} = (\boldsymbol v^{n+1}-S_h
\boldsymbol v^{n+1})+(S_h \boldsymbol v^{n+1}-\boldsymbol v_h^{n+1}) = \boldsymbol \theta_F^{n+1}+\boldsymbol \delta_F^{n+1},
\label{error1}\\
e_P^{n+1} & =p^{n+1}-p_h^{n+1} = (p^{n+1}-\Pi_h p^{n+1})+(\Pi_h p^{n+1}-p_h^{n+1})
= \theta_P^{n+1}+\delta_P^{n+1}, 
\label{errorpom2}\\
\boldsymbol e_\eta^{n+1} & =\boldsymbol \eta^{n+1}-\boldsymbol \eta_h^{n+1} = (\boldsymbol \eta^{n+1}-R_h \boldsymbol \eta^{n+1})+(R_h \boldsymbol \eta^{n+1}-\boldsymbol
\eta_h^{n+1}) = \boldsymbol \theta_{\eta}^{n+1}+ \boldsymbol \delta_{\eta}^{n+1},\\
\boldsymbol e_{\xi}^{n+1} & =\boldsymbol \xi^{n+1}-\boldsymbol \xi_h^{n+1} = (\boldsymbol \xi^{n+1}-P_h \boldsymbol \xi^{n+1})+(P_h \boldsymbol \xi^{n+1}-\boldsymbol
\xi_h^{n+1}) = \boldsymbol \theta_{\xi}^{n+1}+\boldsymbol \delta_{\xi}^{n+1}.
\label{error2}
\end{align}
The main result of this section is stated in the following theorem.

\begin{theorem}\label{MainThm}
Consider the solution $(\boldsymbol \xi_h, \boldsymbol \eta_h, \boldsymbol v_h, p_h)$ of~\eqref{SSweak}-\eqref{SFweak}, with discrete initial data
given by $(\boldsymbol \xi_h^0, \boldsymbol \eta_h^0, \boldsymbol v_h^0, p_h^0) = (P_h \boldsymbol \xi^0, R_h \boldsymbol \eta^0, S_h \boldsymbol v^0, \Pi_h p^0)$.   Assume that the exact solution satisfies assumptions~\eqref{reg1}-\eqref{reg2} and that the following inequality is satisfied:
\begin{gather}
\Delta t \leq \frac{\rho_F}{\alpha C_{TI} k^2} h.
\label{CFLconv}
\end{gather}
 Then, the following estimate holds:
  \begin{align}
& \frac{\rho_F}{2} \|\boldsymbol e_F^{N}  \|^2_{L^2(\Omega_F)} 
+\frac{\rho_S}{2} \|\boldsymbol e_\xi^{N}   \|^2_{L^2(\Omega_S)} 
+\frac12 \| \boldsymbol e_{\eta}^{N} \|^2_{S}
+ \frac{\alpha \Delta t}{2}    \| \boldsymbol e_F^{N} \|^2_{L^2(\Gamma)} 
+ \mu_F \Delta t \sum_{n=0}^{N-1} \| \boldsymbol{D} ( \boldsymbol e_F^{N} ) \|^2_{L^2(\Omega_F)}
\notag \\
&
  \ltt
e^T \left(
h^{2k+2}  \mathcal{A}_0
+h^{2r+2} \mathcal{A}_1
+h^{2k}  \mathcal{A}_2
+\Delta t^2 h^{2k+2} \mathcal{A}_3
+\Delta t^2 \mathcal{A}_4
+\Delta t \mathcal{A}_{5}
    \right),
\notag 
\end{align}
where
\begin{align}
\mathcal{A}_0 & =  \rho_S \| \boldsymbol \xi \|^2_{L^{\infty}(0,T; H^{k+1}(\Omega_S))} 
   +\rho_S \| \partial_t \boldsymbol \xi \|^2_{L^2(0,T; H^{k+1}(\Omega_S))},
 \notag \\
\mathcal{A}_1 & =    \frac{1 }{\mu_F}   \|p\|^2_{L^2(0,T; H^{r+1}(\Omega_F))},
  \notag \\
  \mathcal{A}_2 & = 
\rho_F \| \boldsymbol v \|^2_{L^{\infty}(0,T; H^{k+1}(\Omega_F))}
 + \| \boldsymbol \eta \|^2_{L^{\infty}(0,T; H^{k+1}(\Omega_S))}
 + \| \boldsymbol {\xi} \|^2_{L^2(0,T: H^{k+1}(\Omega_S))}
  \notag \\
  &
 \; \;\; 
  +\frac{ \rho_F^2}{\mu_F} \| \partial_t  \boldsymbol v \|^2_{L^2(0,T: H^{k+1}(\Omega_F))} 
+ \mu_F \| \boldsymbol v \|^2_{L^2(0,T; H^{k+1}(\Omega_F))},
       \notag \\
       \mathcal{A}_3 & =     \left(\frac{\alpha^2  }{\mu_F}+\alpha \right)  \| \partial_t \boldsymbol v \|^2_{L^2(0,T; H^{k+1}(\Gamma))},
\notag \\ 
\mathcal{A}_4 & =  
\frac{ \rho_F^2 }{ \mu_F} \| \partial_{tt} \boldsymbol v \|^2_{L^2(0,T; L^2(\Omega_F))} 
 + \rho_S  \| \partial_{tt} \boldsymbol \xi \|^2_{L^2(0,T; L^2(\Omega_S))} 
+\alpha  \left(\frac{\alpha}{2 \mu_F}+1\right) \| \partial_t \boldsymbol v \|^2_{L^2(0,T; L^2(\Gamma))}
\notag \\
&
\;\;\; + \frac{1}{\alpha} \| \partial_t \boldsymbol \sigma_F  \boldsymbol n_F  \|^2_{L^2(0,T: L^2(\Gamma))}
+ \|\partial_{tt} \boldsymbol \eta\|^2_{L^2(0,T; S)},
 \notag \\
\mathcal{A}_{5} & =   \frac{1}{\alpha}  \| \partial_t  \boldsymbol \sigma_F \boldsymbol n_F \|^2_{L^2(0,T; L^2(\Gamma))}.
\notag
\end{align}

\if 1=0
\begin{align}
& \frac{\rho_F}{2} \|\boldsymbol v^{N} - \boldsymbol v_h^{N} \|^2_{L^2(\Omega_F)} 
+\frac{\rho_S}{2} \|\boldsymbol \xi^{N}   - \boldsymbol \xi_h^{N} \|^2_{L^2(\Omega_S)} 
+\frac12 \| \boldsymbol {\eta}^{N} - \boldsymbol {\eta}_h^{N}\|^2_{S}
+ \frac{\alpha \Delta t}{2}    \| \boldsymbol v^{N} - \boldsymbol v_h^{N} \|^2_{L^2(\Gamma)} 
\notag \\
&
+ \mu_F \Delta t \sum_{n=0}^{N-1} \| \boldsymbol{D} ( \boldsymbol v^{N} - \boldsymbol v_h^{N} ) \|^2_{L^2(\Omega_F)}
\notag \\
& 
 \ltt
e^T
\left( h^{2k+2} \left(
  \rho_S \| \boldsymbol \xi \|^2_{L^{\infty}(0,T; H^{k+1}(\Omega_S))} 
   +\rho_S \| \partial_t \boldsymbol \xi \|^2_{L^2(0,T; H^{k+1}(\Omega_S))}
 \right)
 \right.
 \notag \\
 &
+ h^{2s+2}
 \left( \frac{1 }{\mu_F}   \|p\|^2_{L^2(0,T; H^{s+1}(\Omega_F))}
  +\frac{1 }{\alpha} \| p \|^2_{L^2(0,T; H^{s+1}(\Gamma))}
  \right)
  \notag \\
  &
+ h^{2k} \left(
\rho_F \| \boldsymbol v \|^2_{L^{\infty}(0,T; H^{k+1}(\Omega_F))}
 + \| \boldsymbol \eta \|^2_{L^{\infty}(0,T; H^{k+1}(\Omega_S))}
 + \| \boldsymbol {\xi} \|^2_{L^2(0,T: H^{k+1}(\Omega_S))}
 \right.
  \notag \\
  &
 \; \;\; \left.
  +\frac{ \rho_F^2}{\mu_F} \| \partial_t  \boldsymbol v \|^2_{L^2(0,T: H^{k+1}(\Omega_F))} 
+ \mu_F \| \boldsymbol v \|^2_{L^2(0,T; H^{k+1}(\Omega_F))}
 +\frac{ \mu_F^2}{ \alpha}  \| \boldsymbol v \|^2_{L^2(0,T; H^{k+1}(\Gamma))}
 \right)
\notag 
\\
& 
+\Delta t^2 \left(
\frac{ \rho_F^2 }{ \mu_F} \| \partial_{tt} \boldsymbol v \|^2_{L^2(0,T; L^2(\Omega_F))} 
 + \rho_S  \| \partial_{tt} \boldsymbol \xi \|^2_{L^2(0,T; L^2(\Omega_S))} 
+\alpha  \left(\frac{\alpha}{2 \mu}+1\right) \| \partial_t \boldsymbol v \|^2_{L^2(0,T; L^2(\Gamma))}
\right.
\notag \\
&
\left.
\;\;\; + \frac{1}{\alpha} \| \partial_t \boldsymbol \sigma_F  \boldsymbol n_F  \|^2_{L^2(0,T: L^2(\Gamma))}
+ \|\partial_{tt} \boldsymbol \eta\|^2_{L^2(0,T; L^2(S))} 
\right)
\notag \\
&
   + \Delta t^2  h^{2k+2} \left(\frac{\alpha^2  }{\mu}+\alpha \right)  \| \partial_t \boldsymbol v \|^2_{L^2(0,T; H^{k+1}(\Gamma))}
    + \frac{\Delta t^2 h^{2s+2}}{\alpha} \|\partial_t p \|^2_{L^2(0,T;H^{s+1}(\Gamma))}
       +\Delta t  h^{2k+2} \alpha\| \boldsymbol \xi \|^2_{L^{\infty}(0,T; H^{k+1}(\Gamma))} 
 \notag \\
 &  \;    
 +\frac{\Delta t^2 h^{2k}  \mu_F^2}{\alpha} \|\partial_t \boldsymbol v \|^2_{L^2(0,T;H^{k+1}(\Gamma))}
 + \frac{\Delta t h^{2s+2}}{\alpha} \|\partial_t p \|^2_{L^2(0,T;H^{s+1}(\Gamma))}
 +\frac{\Delta t h^{2k} \mu_F^2}{\alpha} \|\partial_t \boldsymbol v \|^2_{L^2(0,T;H^{k+1}(\Gamma))}
 \notag
\\
& 
\left.
+\frac{\Delta t}{\alpha}  \| \partial_t  \boldsymbol \sigma_F \boldsymbol n_F \|^2_{L^2(0,T; L^2(\Gamma))}
\right)
\end{align}
\fi
\end{theorem}
\begin{proof}
Rearranging the error equation~\eqref{erroreq}, using $\boldsymbol \theta_F^{n+1} = \boldsymbol \theta_{\xi}^{n+1}$ on $\Gamma$, and taking the property~\eqref{Ritz} of the Ritz projection operator into account, we obtain
\begin{align}
& \rho_F \int_{\Omega_F} d_t \boldsymbol \delta_F^{n+1} \cdot \boldsymbol \phi_h 
+a_F(\boldsymbol \delta_F^{n+1}, \boldsymbol \phi_h)
-b_F (\delta_P^{n+1}, \boldsymbol \phi_h)
-b_F( \psi_h, \boldsymbol v_h^{n+1})
+ \rho_S \int_{\Omega_S}  d_t \boldsymbol \delta_{\xi}^{n+1}  \cdot \boldsymbol \zeta_h 
\notag \\
& \; 
+a_S( \boldsymbol \delta_{\eta}^{n+1}, \boldsymbol \zeta_h)
+\alpha \int_{\Gamma} (\boldsymbol \delta_{\xi}^{n+1} - \boldsymbol \delta_F^{n}) \cdot \boldsymbol \zeta_h d \boldsymbol x
+\alpha \int_{\Gamma}(\boldsymbol \delta_F^{n+1} -\boldsymbol \delta_{\xi}^{n+1}) \cdot \boldsymbol \phi_h d \boldsymbol x
\notag \\
&
 =  \int_{\Gamma} \boldsymbol \sigma_F (\boldsymbol e_F^{n} ,e_P^n) \boldsymbol n_F \cdot (\boldsymbol \phi_h - \boldsymbol \zeta_h) 
 -\rho_F \int_{\Omega_F} d_t  \boldsymbol \theta_F^{n+1} \cdot \boldsymbol \phi_h 
 -a_F(\boldsymbol \theta_F^{n+1}, \boldsymbol \phi_h)
 \notag 
 \\
 &
  +b_F (\theta_P^{n+1}, \boldsymbol \phi_h)
 -\rho_S \int_{\Omega_S}  d_t \boldsymbol \theta_{\xi}^{n+1}  \cdot \boldsymbol \zeta_h 
-\alpha \int_{\Gamma} (\boldsymbol \theta_{\xi}^{n+1} - \boldsymbol \theta_F^{n}) \cdot \boldsymbol \zeta_h d \boldsymbol x
+\mathcal{R}_1 (\boldsymbol \phi_h, \boldsymbol \zeta_h).
 \label{error_s}
\end{align}
Let $\boldsymbol \phi_h = \Delta t\boldsymbol \delta_F^{n+1}, \boldsymbol \zeta_h =\Delta t \boldsymbol \delta_{\xi}^{n+1}$ and  $\psi_h = \Delta t\delta_P^{n+1}$. Thanks to~\eqref{press_proj}, the pressure terms simplify as follows:
\begin{gather*}
- \Delta t  b_F(\delta_P^{n+1}, \boldsymbol \delta_F^{n+1})
-\Delta t  b_F(\delta_P^{n+1}, \boldsymbol v_h^{n+1})
 = -\Delta t  b_F(\delta_P^{n+1}, S_h \boldsymbol v^{n+1})= 0.
\end{gather*}
Equation~\eqref{error_s} now becomes
\begin{align}
& \frac{\rho_F}{2}\left( \|\boldsymbol \delta_F^{n+1} \|^2_{L^2(\Omega_F)} 
-  \|\boldsymbol \delta_F^{n} \|^2_{L^2(\Omega_F)} 
+  \|\boldsymbol \delta_F^{n+1}-\boldsymbol \delta_F^{n} \|^2_{L^2(\Omega_F)}\right)
+2 \mu_F \Delta t \| {\boldsymbol D}( \boldsymbol \delta_F^{n+1})\|^2_{L^2(\Omega_F)}
\notag \\
& \;
+\frac{\rho_S}{2}\left( \|\boldsymbol \delta_\xi^{n+1} \|^2_{L^2(\Omega_S)} 
-  \|\boldsymbol \delta_\xi^{n} \|^2_{L^2(\Omega_S)} 
+  \|\boldsymbol \delta_\xi^{n+1}-\boldsymbol \delta_\xi^{n} \|^2_{L^2(\Omega_S)}\right)
+\Delta t a_S( \boldsymbol \delta_{\eta}^{n+1}, \boldsymbol \delta_\xi^{n+1})
\notag \\
& \;
+ \frac{\alpha \Delta t }{2} \left(   \| \boldsymbol \delta_{F}^{n+1} \|^2_{L^2(\Gamma)}  - \|  \boldsymbol \delta_{F}^{n} \|^2_{L^2(\Gamma)} 
+ \|  \boldsymbol \delta_{\xi}^{n+1} - \boldsymbol \delta_{F}^{n} \|^2_{L^2(\Gamma)}
 + \| \boldsymbol \delta_{F}^{n+1} - \boldsymbol \delta_{\xi}^{n+1} \|^2_{L^2(\Gamma)}
  \right) 
\notag \\
& 
 = \Delta t \int_{\Gamma} \boldsymbol \sigma_F (\boldsymbol e_F^n, e_P^{n}) \boldsymbol n_F \cdot (\boldsymbol \delta_{F}^{n+1} - \boldsymbol \delta_{\xi}^{n+1} ) 
-\Delta t\rho_F \int_{\Omega_F} d_t \boldsymbol \theta_F^{n+1} \cdot \boldsymbol \delta_{F}^{n+1}
 \notag \\
 & \;
 -  \Delta t a_F( \boldsymbol \theta_F^{n+1}, \boldsymbol \delta_{F}^{n+1})
+\Delta t b_F(\theta_P^{n+1}, \boldsymbol  \delta_{F}^{n+1})
-\Delta t \rho_S \int_{\Omega_S}  d_t \boldsymbol \theta_{\xi}^{n+1}  \cdot \boldsymbol \delta_{\xi}^{n+1}
\notag \\
& \;
-\alpha \Delta t \int_{\Gamma} (\boldsymbol \theta_{\xi}^{n+1} - \boldsymbol \theta_F^{n}) \cdot \boldsymbol \delta_{\xi}^{n+1}
 +\Delta t \mathcal{R}_1 (\boldsymbol \delta_{F}^{n+1}, \boldsymbol \delta_{\xi}^{n+1}).
 \label{error_eng}
\end{align}
For  term $\Delta t a_S( \boldsymbol \delta_{\eta}^{n+1}, \boldsymbol \delta_{\xi}^{n+1})$ we proceed as follows: 
 \begin{gather*}
\Delta t a_S( \boldsymbol \delta_{\eta}^{n+1}, \boldsymbol \delta_{\xi}^{n+1})= \Delta t a_S(\boldsymbol \delta_{\eta}^{n+1}, d_t \boldsymbol \delta_{\eta}^{n+1}+P_h \boldsymbol{\xi}^{n+1}-R_h d_t \boldsymbol \eta^{n+1})
=\frac12 \| \boldsymbol \delta_{\eta}^{n+1}\|^2_{S}
 -\frac12 \| \boldsymbol \delta_{\eta}^{n}\|^2_{S}  
   \notag \\
    +\frac{\Delta t^2}{2} \| d_t \boldsymbol \delta_{\eta}^{n+1}\|^2_S
 + \Delta t a_S(\boldsymbol \delta_{\eta}^{n+1}, P_h \boldsymbol{\xi}^{n+1}-R_h d_t \boldsymbol \eta^{n+1}).
 \end{gather*}  
 Note that $P_h \boldsymbol{\xi}^{n+1}-R_h d_t \boldsymbol \eta^{n+1} = P_h \boldsymbol \xi^{n+1}-\boldsymbol \xi^{n+1}+\boldsymbol \xi^{n+1}-R_h d_t \boldsymbol \eta^{n+1} = -\boldsymbol{\theta}_{\xi}^{n+1}+d_t \boldsymbol{\theta}_{\eta}^{n+1}+\partial_t \boldsymbol \eta^{n+1}-d_t \boldsymbol \eta^{n+1}.$ Hence, using property~\eqref{Ritz} of the Ritz projection operator, Cauchy-Schwartz and Young's inequalities, we have
 \begin{align}
\Delta t a_S( \boldsymbol \delta_{\eta}^{n+1}, P_h \boldsymbol{\xi}^{n+1}-R_h d_t \boldsymbol \eta^{n+1})
 \leq 
 \Delta t \| \boldsymbol{\theta}_{\xi}^{n+1} \|^2_S+\frac{\Delta t }{4} \|\boldsymbol \delta_{\eta}^{n+1}\|^2_S+\Delta t  \mathcal{R}_2(\boldsymbol \delta_{\eta}^{n+1}),
 \end{align}
 where $\mathcal{R}_2(\boldsymbol \delta_{\eta}^{n+1}) = a_S(\boldsymbol\delta_{\eta}^{n+1}, \partial_t \boldsymbol \eta^{n+1}-d_t \boldsymbol \eta^{n+1})$.

 To estimate the first term on the right hand side of~\eqref{error_eng}, similarly as in~\cite{bukavc2016stability}, we note that $\boldsymbol \delta_{F}^{n+1} -\boldsymbol \delta_{\xi}^{n+1}  = -(\boldsymbol v_h^{n+1}-\boldsymbol \xi_h^{n+1})$ on $\Gamma$. Furthermore, adding and subtracting the  continuous velocity and pressure in~\eqref{Sfluid}, the following relation holds on $\Gamma$:
 \begin{align}
&   \boldsymbol \delta_{F}^{n+1}- \boldsymbol \delta_{\xi}^{n+1}  
= \frac{1}{\alpha} \left(  \bsigma_F(\boldsymbol e_F^{n}, e_P^{n}) \boldsymbol n_F  
-  \bsigma_F(\boldsymbol e_F^{n+1}, e_P^{n+1})\boldsymbol n_F
+  \bsigma_F(\boldsymbol v^{n+1} -\boldsymbol v^{n}, p^{n+1} -p^{n})\boldsymbol n_F
 \right). 
 \label{trstres}
\end{align}
Employing  identity~\eqref{trstres}, we have
  \begin{align}
& \Delta t \int_{\Gamma} \boldsymbol \sigma_F (\boldsymbol e_F^{n}, e_P^{n}) \boldsymbol n_F \cdot (\boldsymbol \delta_{F}^{n+1} - \boldsymbol \delta_{\xi}^{n+1} ) 
 \notag \\
 & \quad = \underbrace{  \frac{\Delta t}{\alpha}  \int_{\Gamma}\boldsymbol \sigma_F (\boldsymbol e_F^{n}, e_P^{n}) \boldsymbol n_F \cdot  \left(  \bsigma_F(\boldsymbol e_F^{n}, e_P^{n}) \boldsymbol n_F  
-  \bsigma_F(\boldsymbol e_F^{n+1}, e_P^{n+1})\boldsymbol n_F \right) }_{\mathcal{T}_1} 
 \notag \\
& \quad \;  +\underbrace{ \frac{\Delta t }{\alpha}  \int_{\Gamma} \boldsymbol \sigma_F (\boldsymbol e_F^{n}, e_P^{n}) \boldsymbol n_F \cdot   \bsigma_F(\boldsymbol v^{n+1} -\boldsymbol v^{n}, p^{n+1} -p^{n})\boldsymbol n_F.}_{\mathcal{T}_2}  
 \end{align}
Using the polarized identity, $\mathcal{T}_1$ is given  as 
\begin{align}
\mathcal{T}_1 &= - \frac{\Delta t}{2 \alpha}    \| \boldsymbol \sigma_F(\boldsymbol e_F^{n+1}, e_P^{n+1}) \boldsymbol n_F \|^2_{L^2(\Gamma)}
  +    \frac{\Delta t}{2 \alpha}  \| \boldsymbol \sigma_F(\boldsymbol e_F^{n}, e_P^{n}) \boldsymbol n_F  \|^2_{L^2(\Gamma)}
   \notag \\
  & \;  +\frac{\Delta t}{2 \alpha}  \| \boldsymbol \sigma_F(\boldsymbol e_F^{n+1}, e_P^{n+1}) \boldsymbol n_F -\boldsymbol \sigma_F(\boldsymbol e_F^{n}, e_P^{n}) \boldsymbol n_F  \|^2_{L^2(\Gamma)}.
  \label{t1}
\end{align}
To estimate the last term in~\eqref{t1}, we again use identity~\eqref{trstres} and Young's inequality as follows:
\begin{align}
  &  \frac{\Delta t}{2 \alpha}  \left\| \boldsymbol \sigma_F(\boldsymbol e_F^{n+1}, e_P^{n+1}) \boldsymbol n_F -\boldsymbol \sigma_F(\boldsymbol e_F^{n}, e_P^{n}) \boldsymbol n_F  \right\|^2_{L^2(\Gamma)}
    \notag \\
 & \quad  =
       \frac{\Delta t}{2 \alpha}   \left\|  \bsigma_F(\boldsymbol v^{n+1} -\boldsymbol v^{n}, p^{n+1} -p^{n})\boldsymbol n_F
     -\alpha (\boldsymbol \delta_{F}^{n+1}-\boldsymbol \delta_{\xi}^{n+1})  \right\|^2_{L^2(\Gamma)}
            \notag \\
 & \quad
       =       \frac{\Delta t}{2 \alpha}  \left\| \bsigma_F(\boldsymbol v^{n+1} -\boldsymbol v^{n}, p^{n+1} -p^{n})\boldsymbol n_F  \right\|^2_{L^2(\Gamma)} 
       +   \frac{\alpha \Delta t }{2}  \|  \boldsymbol \delta_{F}^{n+1}-\boldsymbol \delta_{\xi}^{n+1}  \|^2_{L^2(\Gamma)}
    \notag   \\
&        \quad \;      - \Delta t \int_{\Gamma} (\boldsymbol \delta_{F}^{n+1}  -\boldsymbol \delta_{\xi}^{n+1}  ) \cdot \bsigma_F(\boldsymbol v^{n+1} -\boldsymbol v^{n}, p^{n+1} -p^{n})\boldsymbol n_F
            \notag \\
  &  \quad    \leq    
           \frac{\Delta t}{2 \alpha}   \left\| \bsigma_F(\boldsymbol v^{n+1} -\boldsymbol v^{n}, p^{n+1} -p^{n})\boldsymbol n_F  \right\|^2_{L^2(\Gamma)} 
       +   \frac{\alpha \Delta t}{2}  \| \boldsymbol \delta_{F}^{n+1}-\boldsymbol \delta_{\xi}^{n+1} \|^2_{L^2(\Gamma)}
    \notag   \\
&   \quad      
  +\frac{\alpha \Delta t}{12}  \| \boldsymbol \delta_{F}^{n+1}-\boldsymbol \delta_{\xi}^{n+1} \|^2_{L^2(\Gamma)}
                +\frac{3\Delta t}{ \alpha}  \left\| \bsigma_F(\boldsymbol v^{n+1} -\boldsymbol v^{n}, p^{n+1} -p^{n})\boldsymbol n_F \right\|^2_{L^2(\Gamma)}.
\end{align}
Finally, we estimate $\mathcal{T}_2$ using the Cauchy-Schwartz inequality and Young's inequality  as
\begin{align}
\mathcal{T}_2 & 
 \leq
\frac{ \Delta t^2}{2 \alpha }  \left\|\boldsymbol \sigma_F(\boldsymbol e_F^{n}, e_P^{n}) \boldsymbol n_F \right\|^2_{L^2(\Gamma)}
+\frac{  1}{ 2 \alpha } \left\| \boldsymbol \sigma_F\left( \boldsymbol v^{n+1}-\boldsymbol v^{n}, p^{n+1}-p^n \right) 
\boldsymbol n_F \right\|^2_{L^2(\Gamma)}.
\end{align}
We bound the remaining terms in~\eqref{error_eng} as follows. Using Cauchy-Schwartz, Young's, Poincar\'e - Friedrichs, and Korn's inequalities, we have
 \begin{align*}
&-\Delta t\rho_F \int_{\Omega_F} d_t \boldsymbol \theta_F^{n+1} \cdot \boldsymbol \delta_{F}^{n+1}
- \Delta t a_F ( \boldsymbol \theta_F^{n+1}, \boldsymbol \delta_{F}^{n+1})
+\Delta t b_F (\theta_P^{n+1},  \boldsymbol  \delta_{F}^{n+1}) 
-\Delta t \rho_S \int_{\Omega_S}  d_t \boldsymbol \theta_{\xi}^{n+1}  \cdot \boldsymbol \delta_{\xi}^{n+1}
\\
&\quad   \ltt \frac{ \Delta t \rho_F^2}{\mu_F} \| d_t  \boldsymbol \theta_F^{n+1}\|^2_{L^2(\Omega_F)} 
+ \Delta t \mu_F \| \boldsymbol{D} (\boldsymbol \theta_F^{n+1}) \|^2_{L^2(\Omega_F)}+\frac{ \Delta t }{\mu_F}  \|\theta_P^{n+1}\|^2_{L^2(\Omega_F)}
+\frac{\mu_F \Delta t}{4}\|\boldsymbol D (\boldsymbol \delta_F^{n+1}) \|^2_{L^2(\Omega_F)}
\notag
\\
&\qquad 
+\Delta t \rho_S \| d_t \boldsymbol \theta_\xi^{n+1} \|^2_{L^2(\Omega_S)}
+\frac{\Delta t \rho_S }{4} \| \boldsymbol \delta_\xi^{n+1}\|^2_{L^2(\Omega_S)}.
\end{align*}
Next, noting that $\boldsymbol \theta_F^{n+1} = \boldsymbol \theta_{\xi}^{n+1}$ on $\Gamma$ and adding and subtracting $\boldsymbol \delta_F^{n+1}$,  we have
\begin{align}
&-\alpha \Delta t \int_{\Gamma} (\boldsymbol \theta_{\xi}^{n+1} - \boldsymbol \theta_F^{n}) \cdot \boldsymbol \delta_{\xi}^{n+1}
\notag \\
& \quad 
=
-\alpha \Delta t \int_{\Gamma}(\boldsymbol \theta_F^{n+1} -\boldsymbol \theta_{F}^{n}) \cdot \boldsymbol \delta_F^{n+1}
-\alpha \Delta t \int_{\Gamma}(\boldsymbol \theta_F^{n+1} -\boldsymbol \theta_{F}^{n}) \cdot (\boldsymbol \delta_\xi^{n+1}-\boldsymbol \delta_F^{n+1})
\notag \\
&\quad \ltt
\Delta t^3 \left(\frac{\alpha^2 }{\mu_F}+ \alpha \right) \|d_t \boldsymbol \theta_F^{n+1} \|^2_{L^2(\Gamma)}
+\frac{ \mu_F \Delta t }{4} \|\boldsymbol D(\boldsymbol \delta_F^{n+1}) \|^2_{L^2(\Omega_F)}
+\frac{\alpha \Delta t }{12} \|\boldsymbol \delta_F^{n+1}-\boldsymbol \delta_\xi^{n+1} \|^2_{L^2(\Gamma)}.
\notag
\end{align}

Combining the estimates above with equation~\eqref{error_eng}, summing from $n=0, \ldots, N-1$  and taking into account the assumption on the initial data, we have 
\begin{align}
& \frac{\rho_F}{2} \|\boldsymbol \delta_F^{N} \|^2_{L^2(\Omega_F)} 
+\frac{\rho_S}{2} \|\boldsymbol \delta_\xi^{N} \|^2_{L^2(\Omega_S)} 
+\frac12 \| \boldsymbol \delta_{\eta}^{N}\|^2_{S}
+ \frac{\alpha \Delta t}{2}    \| \boldsymbol \delta_{F}^{N} \|^2_{L^2(\Gamma)} 
+  \frac{\Delta t}{2 \alpha}    \| \boldsymbol \sigma_F(\boldsymbol e_F^{N}, e_P^{N}) \boldsymbol n_F \|^2_{L^2(\Gamma)}
\notag \\
& \;
+\frac{3}{2} \mu_F \Delta t \sum_{n=0}^{N-1} \| {\boldsymbol D}( \boldsymbol \delta_F^{n+1})\|^2_{L^2(\Omega_F)}
+  \frac{\rho_F \Delta t^2}{2} \sum_{n=0}^{N-1} \|d_t\boldsymbol \delta_F^{n+1} \|^2_{L^2(\Omega_F)}
+\frac{\rho_S \Delta t^2}{2}  \sum_{n=0}^{N-1} \| d_t \boldsymbol \delta_\xi^{n+1} \|^2_{L^2(\Omega_S)}
\notag \\
& \;
    +\frac{\Delta t^2}{2}  \sum_{n=0}^{N-1} \| d_t \boldsymbol \delta_{\eta}^{n+1}\|^2_S
+ \frac{\alpha \Delta t}{2}  \sum_{n=0}^{N-1}
 \|  \boldsymbol \delta_{\xi}^{n+1} - \boldsymbol \delta_{F}^{n} \|^2_{L^2(\Gamma)}
\notag \\
& 
\ltt
 \Delta t \sum_{n=0}^{N-1} \| \boldsymbol{\theta}_{\xi}^{n+1} \|^2_S
  +\frac{ \Delta t \rho_F^2}{\mu_F} \sum_{n=0}^{N-1} \| d_t  \boldsymbol \theta_F^{n+1}\|^2_{L^2(\Omega_F)} 
+ \Delta t \mu_F \sum_{n=0}^{N-1} \| \boldsymbol{D} (\boldsymbol \theta_F^{n+1}) \|^2_{L^2(\Omega_F)}
+\frac{ \Delta t }{\mu_F}  \sum_{n=0}^{N-1} \|\theta_P^{n+1}\|^2_{L^2(\Omega_F)}
 \notag  \\
 & \;
 +\Delta t \rho_S \sum_{n=0}^{N-1} \| d_t \boldsymbol \theta_\xi^{n+1} \|^2_{L^2(\Omega_S)}
  + \Delta t^3 \left(\frac{\alpha^2 }{\mu_F}+ \alpha  \right) \sum_{n=0}^{N-1}  \|d_t \boldsymbol \theta_F^{n+1} \|^2_{L^2(\Gamma)}
 \notag \\
 &  \;    
 + \frac{\Delta t+1}{\alpha}   \sum_{n=0}^{N-1}  \left\|\boldsymbol \sigma_F\left(\boldsymbol v^{n+1}-\boldsymbol v^{n}, p^{n+1}-p^n \right) \boldsymbol n_F  \right\|^2_{L^2(\Gamma)} 
    \notag   \\
&   \;      
+\frac{\Delta t^2}{2 \alpha }  \sum_{n=0}^{N-1} \left\|\boldsymbol \sigma_F(\boldsymbol e_F^{n}, e_P^{n}) \boldsymbol n_F \right\|^2_{L^2(\Gamma)}
   +\frac{ \alpha \Delta t}{6}  \sum_{n=0}^{N-1} \| \boldsymbol \delta_{F}^{n+1}-\boldsymbol \delta_{\xi}^{n+1} \|^2_{L^2(\Gamma)}
   +\frac{\Delta t \rho_S }{4}  \sum_{n=0}^{N-1} \| \boldsymbol \delta_\xi^{n+1}\|^2_{L^2(\Omega_S)}
\notag \\
&\;
+\frac{\Delta t}{4}  \sum_{n=0}^{N-1} \|\boldsymbol \delta_{\eta}^{n+1}\|^2_S
   +\Delta t \sum_{n=0}^{N-1} \mathcal{R}_1 (\boldsymbol \delta_{F}^{n+1}, \boldsymbol \delta_{\xi}^{n+1})
   +\Delta t   \sum_{n=0}^{N-1}\mathcal{R}_2(\boldsymbol \delta_{\eta}^{n+1}).
 \label{error_eng2}
\end{align}

To estimate the approximation and consistency errors, we use Lemmas~\ref{cons1} and~\ref{lemma_interpolation}, leading to the following inequality:
\begin{align}
& \frac{\rho_F}{2} \|\boldsymbol \delta_F^{N} \|^2_{L^2(\Omega_F)} 
+\frac{\rho_S}{2} \|\boldsymbol \delta_\xi^{N} \|^2_{L^2(\Omega_S)} 
+\frac12 \| \boldsymbol \delta_{\eta}^{N}\|^2_{S}
+ \frac{\alpha \Delta t}{2}    \| \boldsymbol \delta_{F}^{N} \|^2_{L^2(\Gamma)} 
+  \frac{\Delta t}{2 \alpha}    \| \boldsymbol \sigma_F(\boldsymbol e_F^{N}, e_P^{N}) \boldsymbol n_F \|^2_{L^2(\Gamma)}
\notag \\
& \;
+ \mu_F \Delta t \sum_{n=0}^{N-1} \| {\boldsymbol D}( \boldsymbol \delta_F^{n+1})\|^2_{L^2(\Omega_F)}
+  \frac{\rho_F \Delta t^2}{2}  \sum_{n=0}^{N-1} \|d_t\boldsymbol \delta_F^{n+1} \|^2_{L^2(\Omega_F)}
+\frac{\rho_S \Delta t^2}{2}  \sum_{n=0}^{N-1} \| d_t \boldsymbol \delta_\xi^{n+1} \|^2_{L^2(\Omega_S)}
\notag \\
& \;
    +\frac{\Delta t^2}{2}  \sum_{n=0}^{N-1} \| d_t \boldsymbol \delta_{\eta}^{n+1}\|^2_S
+ \frac{\alpha \Delta t}{2}  \sum_{n=0}^{N-1}
 \|  \boldsymbol \delta_{\xi}^{n+1} - \boldsymbol \delta_{F}^{n} \|^2_{L^2(\Gamma)}
\notag \\
& 
 \ltt
h^{2k} \| \boldsymbol {\xi} \|^2_{L^2(0,T: H^{k+1}(\Omega_S))}
  +\frac{ \rho_F^2}{\mu_F} h^{2k}\| \partial_t  \boldsymbol v \|^2_{L^2(0,T: H^{k+1}(\Omega_F))} 
+ \mu_F h^{2k} \| \boldsymbol v \|^2_{L^2(0,T; H^{k+1}(\Omega_F))}
\notag 
\\
& 
\;
+\frac{1 }{\mu_F} h^{2r+2}  \|p\|^2_{L^2(0,T; H^{r+1}(\Omega_F))}
  +\rho_S h^{2k+2}\| \partial_t \boldsymbol \xi \|^2_{L^2(0,T; H^{k+1}(\Omega_S))}
 \notag  \\
 & \;
    + \Delta t^2 \left(\frac{\alpha^2  }{\mu_F}+\alpha   \right) h^{2k+2}  \| \partial_t \boldsymbol v \|^2_{L^2(0,T; H^{k+1}(\Gamma))}
     +\frac{ \Delta t^2 \rho_F^2 }{ \mu_F} \| \partial_{tt} \boldsymbol v \|^2_{L^2(0,T; L^2(\Omega_F))} 
 \notag\\
&  \;  
 + \Delta t^2 \rho_S  \| \partial_{tt} \boldsymbol \xi \|^2_{L^2(0,T; L^2(\Omega_S))}
+\alpha \Delta t^2 \left(\frac{\alpha}{\mu_F}+1 \right) \| \partial_t \boldsymbol v \|^2_{L^2(0,T; L^2(\Gamma))}
 \notag \\
 & \;  
 + \frac{ \Delta t (\Delta t +1)}{\alpha} \| \partial_t \boldsymbol \sigma_F  \boldsymbol n_F  \|^2_{L^2(0,T: L^2(\Gamma))}
+\Delta t^2  \|\partial_{tt} \boldsymbol \eta\|^2_{L^2(0,T; S)}   
+\frac{  \Delta t^2}{2 \alpha } \sum_{n=0}^{N-1} \left\|\boldsymbol \sigma_F(\boldsymbol e_F^{n}, e_P^{n}) \boldsymbol n_F \right\|^2_{L^2(\Gamma)}
    \notag   \\
&   \;      
   +\frac{\alpha \Delta t}{4} 
   \sum_{n=0}^{N-1} \| \boldsymbol \delta_{F}^{n+1}-\boldsymbol \delta_{\xi}^{n+1} \|^2_{L^2(\Gamma)}
+\frac{\Delta t \rho_S}{2}  \sum_{n=0}^{N-1} \| \boldsymbol \delta_\xi^{n+1}\|^2_{L^2(\Omega_S)}
+ \frac{\Delta t}{2}\sum_{n=0}^{N-1} \| \boldsymbol \delta_{\eta}^{n+1} \|^2_S.
 \label{errorN}
\end{align}
We estimate term $\displaystyle\frac{\alpha \Delta t}{4} \sum_{n=0}^{N-1}  \|\boldsymbol \delta_F^{n+1} - \boldsymbol \delta_{\xi}^{n+1}  \|^2_{L^2(\Gamma)}$ by adding and subtracting $\boldsymbol \delta_F^n$ and using trace-inverse inequality~\eqref{traceinverse} as follows: 
\begin{align}
&
\frac{\alpha \Delta t}{4} \sum_{n=0}^{N-1}  \| \boldsymbol \delta_F^{n+1} - \boldsymbol \delta_{\xi}^{n+1}  \|^2_{L^2(\Gamma)} 
=\frac{\alpha \Delta t}{4} \sum_{n=0}^{N-1}  \| \boldsymbol \delta_F^{n+1}- \boldsymbol \delta_F^{n}+ \boldsymbol \delta_F^{n} - \boldsymbol \delta_{\xi}^{n+1}  \|^2_{L^2(\Gamma)} 
\notag \\
& \quad 
\leq 
\frac{\alpha \Delta t}{2} \sum_{n=0}^{N-1}  \| \boldsymbol \delta_F^{n+1} - \boldsymbol \delta_{F}^{n}  \|^2_{L^2(\Gamma)} 
+\frac{\alpha \Delta t}{2} \sum_{n=0}^{N-1}  \| \boldsymbol \delta_{\xi}^{n+1} -\boldsymbol \delta_F^{n}  \|^2_{L^2(\Gamma)} 
\notag \\
& \quad 
\leq 
\frac{\alpha C_{TI} k^2 \Delta t}{2 h} \sum_{n=0}^{N-1}  \| \boldsymbol \delta_F^{n+1} - \boldsymbol \delta_{F}^{n}  \|^2_{L^2(\Omega_F)} 
+\frac{\alpha \Delta t}{2} \sum_{n=0}^{N-1}  \| \boldsymbol \delta_{\xi}^{n+1} -\boldsymbol \delta_F^{n}  \|^2_{L^2(\Gamma)}. 
\label{deltaDiff}
\end{align}
Combining~\eqref{deltaDiff} with~\eqref{errorN}, we get
\begin{align}
& \frac{\rho_F}{2} \|\boldsymbol \delta_F^{N} \|^2_{L^2(\Omega_F)} 
+\frac{\rho_S}{2} \|\boldsymbol \delta_\xi^{N} \|^2_{L^2(\Omega_S)} 
+\frac12 \| \boldsymbol \delta_{\eta}^{N}\|^2_{S}
+ \frac{\alpha \Delta t}{2}    \| \boldsymbol \delta_{F}^{N} \|^2_{L^2(\Gamma)} 
+  \frac{\Delta t}{2 \alpha}    \| \boldsymbol \sigma_F(\boldsymbol e_F^{N}, e_P^{N}) \boldsymbol n_F \|^2_{L^2(\Gamma)}
\notag \\
& \;
+ \mu_F \Delta t \sum_{n=0}^{N-1} \| {\boldsymbol D}( \boldsymbol \delta_F^{n+1})\|^2_{L^2(\Omega_F)}
+ \frac{\Delta t^2}{2} \left( \rho_F  - \frac{\alpha C_{TI} k^2 \Delta t}{h}  \right) \sum_{n=0}^{N-1} \|d_t\boldsymbol \delta_F^{n+1} \|^2_{L^2(\Omega_F)}
\notag \\
& \;
+\frac{\rho_S \Delta t^2}{2}  \sum_{n=0}^{N-1} \| d_t \boldsymbol \delta_\xi^{n+1} \|^2_{L^2(\Omega_S)}
    +\frac{\Delta t^2}{2}  \sum_{n=0}^{N-1} \| d_t \boldsymbol \delta_{\eta}^{n+1}\|^2_S
\notag \\
& 
 \ltt
h^{2k} \| \boldsymbol {\xi} \|^2_{L^2(0,T: H^{k+1}(\Omega_S))}
  +\frac{ \rho_F^2}{\mu_F} h^{2k}\| \partial_t  \boldsymbol v \|^2_{L^2(0,T: H^{k+1}(\Omega_F))} 
+ \mu_F h^{2k} \| \boldsymbol v \|^2_{L^2(0,T; H^{k+1}(\Omega_F))}
\notag 
\\
& 
\;
+\frac{1 }{\mu_F} h^{2r+2}  \|p\|^2_{L^2(0,T; H^{r+1}(\Omega_F))}
  +\rho_S h^{2k+2}\| \partial_t \boldsymbol \xi \|^2_{L^2(0,T; H^{k+1}(\Omega_S))}
 \notag  \\
 & \;
    + \Delta t^2 \left(\frac{\alpha^2  }{\mu_F}+\alpha   \right) h^{2k+2}  \| \partial_t \boldsymbol v \|^2_{L^2(0,T; H^{k+1}(\Gamma))}
     +\frac{ \Delta t^2 \rho_F^2 }{ \mu_F} \| \partial_{tt} \boldsymbol v \|^2_{L^2(0,T; L^2(\Omega_F))} 
 \notag\\
&  \;  
 + \Delta t^2 \rho_S  \| \partial_{tt} \boldsymbol \xi \|^2_{L^2(0,T; L^2(\Omega_S))}
+\alpha \Delta t^2 \left(\frac{\alpha}{\mu_F}+1 \right) \| \partial_t \boldsymbol v \|^2_{L^2(0,T; L^2(\Gamma))}
 \notag \\
 & \;  
 + \frac{ \Delta t (\Delta t +1)}{\alpha} \| \partial_t \boldsymbol \sigma_F  \boldsymbol n_F  \|^2_{L^2(0,T: L^2(\Gamma))}
+\Delta t^2  \|\partial_{tt} \boldsymbol \eta\|^2_{L^2(0,T; S)}   
+\frac{  \Delta t^2}{2 \alpha } \sum_{n=0}^{N-1} \left\|\boldsymbol \sigma_F(\boldsymbol e_F^{n}, e_P^{n}) \boldsymbol n_F \right\|^2_{L^2(\Gamma)}
    \notag   \\
&   \;      
+\frac{\Delta t \rho_S}{2}  \sum_{n=0}^{N-1} \| \boldsymbol \delta_\xi^{n+1}\|^2_{L^2(\Omega_S)}
+ \frac{\Delta t}{2}\sum_{n=0}^{N-1} \| \boldsymbol \delta_{\eta}^{n+1} \|^2_S.
 \label{errorNF}
\end{align}

We recall that the error between the exact and the discrete solution is the sum of the approximation error and the
truncation error. Thus, using the triangle inequality, approximation properties~\eqref{app1}-\eqref{app2} and the Gronwall lemma, we
prove the desired estimate.
 \end{proof}

Using Taylor-Hood elements, i.e. $k = 2, r = 1$, for the fluid problem and piecewise quadratic elements for the solid problem, we have the following 
estimate.
\begin{corollary}
Consider algorithm~\eqref{SSweak}-\eqref{SFweak}. Suppose that $(V^F_h,Q_h^F)$ is given by $\mathbb{P}_2-\mathbb{P}_1$ Taylor-Hood approximation elements and $V_h^S$ is given by $\mathbb{P}_2$ approximation elements. 
Under the assumptions of Theorem~\ref{MainThm}, we have
   \begin{align}
& \frac{\rho_F}{2} \|\boldsymbol e_F^{N}  \|^2_{L^2(\Omega_F)} 
+\frac{\rho_S}{2} \|\boldsymbol e_\xi^{N}   \|^2_{L^2(\Omega_S)} 
+\frac12 \| \boldsymbol e_{\eta}^{N} \|^2_{S}
+ \frac{\alpha \Delta t}{2}    \| \boldsymbol e_F^{N} \|^2_{L^2(\Gamma)} 
+ \mu_F \Delta t \sum_{n=0}^{N-1} \| \boldsymbol{D} ( \boldsymbol e_F^{N} ) \|^2_{L^2(\Omega_F)}
\notag \\
&
  \ltt
e^T \left(
h^4+\Delta t  \right).
\notag 
\end{align}
\end{corollary}

The following lemmas are used in the proof of Theorem~\ref{MainThm}.                                                              
\begin{lemma} \label{cons1}
The following estimate holds:
\begin{align*}
& \Delta t \sum_{n=0}^{N-1} \big(\mathcal{R}_1 (\boldsymbol \delta_{F}^{n+1}, \boldsymbol \delta_{\xi}^{n+1})+\mathcal{R}_2(\boldsymbol \delta_{\eta}^{n+1}) \big)  \\
&
\ltt
\Delta t^2 \left(  \frac{\rho_F^2 }{ \mu_F}\| \partial_{tt} \boldsymbol v \|^2_{L^2(0,T; L^2(\Omega_F))}
+\rho_S \| \partial_{tt} \boldsymbol \xi \|^2_{L^2(0,T; L^2(\Omega_S))}
+\alpha \left(\frac{\alpha}{\mu_F}+1 \right) \| \partial_t \boldsymbol v\|^2_{L^2(0,T; L^2(\Gamma))}
\right.
 \\
 & \quad 
 \left. 
 + \frac{1}{\alpha}  \| \partial_t  \boldsymbol \sigma_F \boldsymbol n_F \|^2_{L^2(0,T; L^2(\Gamma))}
  + \| \partial_{tt} \boldsymbol \eta\|^2_{L^2(0,T; S)} 
\right)
   + \frac{\mu_F \Delta t}{2} \sum_{n=0}^{N-1} \| \boldsymbol D(\boldsymbol \delta_F^{n+1})\|^2_{L^2(\Omega_F)} 
    \\
 & \quad 
  + \frac{ \Delta t \rho_S }{4}\| \boldsymbol \delta_{\xi}^{n+1} \|^2_{L^2(\Omega_S)} 
  +\frac{\alpha \Delta t}{10}  \| \boldsymbol \delta_{F}^{n+1}- \boldsymbol \delta_{\xi}^{n+1} \|^2_{L^2(\Gamma)}
    +\frac{ \Delta t}{4} \sum_{n=0}^{N-1}   \| \boldsymbol \delta_{\eta}^{n+1}\|^2_S.
\end{align*}
\end{lemma}
\begin{proof}
Rearranging and using Cauchy-Schwartz, Young's, Poincar\'e - Friedrichs, and Korn's inequalities, we have
\begin{align*}
\Delta t\mathcal{R}_1  (\boldsymbol \delta_{F}^{n+1}, \boldsymbol \delta_{\xi}^{n+1}) 
 &
 =
 \Delta t\rho_F \int_{\Omega_F} (d_t \boldsymbol v^{n+1} - \partial_t \boldsymbol v^{n+1}) \cdot \boldsymbol \delta_{F}^{n+1}  
+\Delta t\rho_S \int_{\Omega_S} (d_t \boldsymbol \xi^{n+1} - \partial_t \boldsymbol \xi^{n+1}) \cdot \boldsymbol \delta_{\xi}^{n+1}
\notag \\
& \; 
+\alpha \Delta t \int_{\Gamma} (\boldsymbol v^{n+1}  - \boldsymbol v^{n}) \cdot \boldsymbol \delta_{F}^{n+1} d \boldsymbol x
+\alpha \Delta t \int_{\Gamma} (\boldsymbol v^{n+1}  - \boldsymbol v^{n}) \cdot (\boldsymbol \delta_{\xi}^{n+1} -\boldsymbol \delta_{F}^{n+1}) d \boldsymbol x
\notag,\\
 & \;  
 +\Delta t\int_{\Gamma} \boldsymbol \sigma_F (\boldsymbol v^{n+1} - \boldsymbol v^n,p^{n+1}-p^n) \boldsymbol n_F \cdot  (\boldsymbol \delta_{F}^{n+1} -\boldsymbol \delta_{\xi}^{n+1}) 
 \notag \\
  & \ltt 
  \frac{\Delta t \rho_F^2 }{ \mu_F} \| d_t \boldsymbol v^{n+1}-\partial_t \boldsymbol v^{n+1}\|^2_{L^2(\Omega_F)} 
  +\frac{\mu_F \Delta t}{2}   \| \boldsymbol D(\boldsymbol \delta_F^{n+1})\|^2_{L^2(\Omega_F)} 
\notag\\
&  \; 
 +  \Delta t \rho_S  \| d_t \boldsymbol \xi^{n+1}-\partial_t \boldsymbol \xi^{n+1}\|^2_{L^2(\Omega_S)} 
  + \frac{ \Delta t \rho_S }{4}\| \boldsymbol \delta_{\xi}^{n+1} \|^2_{L^2(\Omega_S)} 
\notag\\
& \;
+\alpha \Delta t \left(\frac{\alpha}{\mu_F}+1 \right) \| \boldsymbol v^{n+1} - \boldsymbol v^n \|^2_{L^2(\Gamma)}
+\frac{\alpha \Delta t}{12}  \| \boldsymbol \delta_{F}^{n+1}- \boldsymbol \delta_{\xi}^{n+1} \|^2_{L^2(\Gamma)}
 \notag \\
 & \;  
 + \frac{ \Delta t}{\alpha} \|\boldsymbol \sigma_F (\boldsymbol v^{n+1} - \boldsymbol v^n,p^{n+1}-p^n) \boldsymbol n_F  \|^2_{L^2(\Gamma)}.
\end{align*}

Furthermore, using Cauchy-Schwartz and Young's inequalities, we have
\begin{align*}
\Delta t \mathcal{R}_2(\boldsymbol \delta_{\eta}^{n+1}) 
&= \Delta t a_S(\boldsymbol\delta_{\eta}^{n+1}, \partial_t \boldsymbol \eta^{n+1}-d_t \boldsymbol \eta^{n+1}) 
\\
&
\leq \Delta t  \| d_t \boldsymbol \eta^{n+1}-\partial_t \boldsymbol \eta^{n+1}\|^2_S 
+\frac{\Delta t}{4}  \| \boldsymbol \delta_{\eta}^{n+1}\|^2_S.
\end{align*}
The final estimate follows by summing from $n=0$ to $N-1$ and applying Lemma~\ref{consistency}.
\end{proof}

\begin{lemma}[Consistency errors] \label{consistency}
Assume $X \in \{\Omega, \Gamma\}$. The following inequalities hold:
\begin{align*}
&\Delta t \sum_{n=0}^{N-1}\| d_t \boldsymbol \varphi^{n+1}-\partial_t \boldsymbol \varphi^{n+1}\|^2_{L^2(X)} \ltt \Delta t^2 \|\partial_{tt} \boldsymbol \varphi\|^2_{L^2(0,T;L^2(X))}, \\
&\Delta t \displaystyle\sum_{n=0}^{N-1}\| \boldsymbol \varphi^{n+1} -\boldsymbol \varphi^{n}   \|^2_{L^2(X)}\ltt \Delta t^2  \| \partial_t  \boldsymbol \varphi  \|^2_{L^2(0,T; L^2(X))}.
\end{align*}
\end{lemma}
\begin{proof}
See~\cite{bukavc2016stability} for proof.
\end{proof}

\begin{lemma}[Interpolation errors] \label{lemma_interpolation} The following inequalities hold:
\begin{gather*}
\Delta t \sum_{n=0}^{N-1} \| d_t \boldsymbol \theta_F^{n+1}\|^2_{L^2(\Omega_F)} \le  \| \partial_t \boldsymbol \theta_F\|^2_{L^2(0,T;L^2(\Omega_F))} \ltt h^{2k} \|\partial_t \boldsymbol v
\|^2_{L^2(0,T;H^{k+1}(\Omega_F))}, \\
\Delta t \sum_{n=0}^{N-1} \| d_t \boldsymbol \theta_{\xi}^{n+1}\|^2_{L^2(\Omega_S)} \le  \|  \partial_t \boldsymbol \theta_{\xi}\|^2_{L^2(0,T;L^2(\Omega_S))} \ltt  h^{2k+2} \|\partial_t \boldsymbol \xi\|^2_{L^2(0,T;H^{k+1}(\Omega_S))}, \\
\Delta t \sum_{n=0}^{N-1} \|\boldsymbol D(\boldsymbol \theta_F^{n+1})\|^2_{L^2(\Omega_F)} \ltt \Delta t \sum_{n=0}^{N-1} h^{2k} \| \boldsymbol v^{n+1}\|^2_{H^{k+1}(\Omega_F)}  \ltt  h^{2k} \| \boldsymbol v\|^2_{L^2(0,T;H^{k+1}(\Omega_F))}, \\
\Delta t \sum_{n=0}^{N-1} \| \boldsymbol \theta_{\eta}^{n+1} \|^2_S \ltt 
 h^{2k} \|\boldsymbol \eta\|^2_{L^2(0,T;H^{k+1}(\Omega_S))}, \qquad
\Delta t \sum_{n=0}^{N-1} \|\theta_p^{n+1}\|^2_{L^2(\Omega_F)} \ltt h^{2r+2} \| p \|^2_{L^2(0,T; H^{r+1}(\Omega_F))}. 
\end{gather*}
\end{lemma}
\begin{proof}
The last three inequalities follow directly from approximation properties~\eqref{app1}-\eqref{app2}. For other inequalities, see~\cite{bukavc2016stability} for more details.
\end{proof}

\begin{remark}
The sub-optimal order of convergence in time that is shown in this paper is often obtained in partitioned methods for the interaction between a fluid and thick structure. In particular, sub-optimal accuracy has been shown for the partitioned  method based on Nitsche's approach in~\cite{burman2009stabilization} and for the Robin-Neumann method in~\cite{fernandez2015generalized}. Extending the algorithm to optimal accuracy could be achieved by using higher-order extrapolations in the design of the generalized Robin coupling conditions, but it is out of scope of this paper.
\end{remark}

\section{Numerical examples}~\label{numerics}
To demonstrate the performance of the proposed numerical scheme, we present three numerical examples.
In the first example, we investigate the accuracy of the linearized FSI problem~\eqref{Ssolid}-\eqref{Sfluid} considered in Section~\ref{conv} and compare the approximated solution to a manufactured one. We consider the same benchmark problem in the second example, but apply it to a moving domain FSI problem~\eqref{fsi1}-\eqref{fsi2}. In both of these examples, the convergence rates are calculated using different combination parameters, $\alpha$, in order to show the theory is satisfied and in some cases, exceeded. In our final example, we model pressure propagation in a two-dimensional channel with physiologically realistic parameters for blood flow  and show the comparison of the results obtained using the proposed partitioned scheme and a monolithic method.

\subsection{Example 1}

In the first numerical example, we use the method of manufactured solutions to verify the theoretical convergence results from Section~\ref{conv}. 
We define the  structure and fluid domains as upper and lower parts of the unit square, respectively, i.e. ${\Omega}_S=(0,1) \times (\frac{1}{2},1)$ and ${\Omega}_F=(0,1) \times (0, \frac{1}{2})$. 
The true solutions for the structure displacement, $\boldsymbol \eta$, the fluid velocity, $\boldsymbol v$, and the fluid pressure, $p$, are defined as:
\begin{align}
 &  \begin{bmatrix}
	\eta_x \\ \eta_y
   \end{bmatrix}
   = 
    \begin{bmatrix}
	10^{-3}2x(1-x)y(1-y)e^t \\ 10^{-3}x(1-x)y(1-y) e^t \label{true_eta}
   \end{bmatrix},
   \\
   &\begin{bmatrix}
   	v_x \\ v_y
   \end{bmatrix}
   =
   \begin{bmatrix}
   	10^{-3}2x(1-x)y(1-y)e^t \\ 10^{-3}x(1-x)y(1-y)e^t \label{true_u}
   \end{bmatrix},
   \\
   &p  =-10^{-3} e^t \lambda_S \left(2(1-2x)y(1-y) + x(1-x)(1-2y)\right). \label{true_p}
\end{align}
We note that the fluid velocity is not divergence-free. Therefore, we add a forcing term to the conservation of mass equation. We also add forcing terms in both the fluid and structure equations~\eqref{Ssolid}-\eqref{Sfluid}, resulting in the following system:
\begin{align*}
& \rho_F \partial_t \boldsymbol{v}  = \nabla \cdot \boldsymbol\sigma_F(\boldsymbol v, p) + \boldsymbol{f}_F&  \textrm{in}\; \Omega_F \times(0,T), 
\\ 
&\nabla \cdot \boldsymbol{v} = s & \textrm{in}\; \Omega_F \times(0,T),
\\ 
& \partial_{t} {\boldsymbol \eta}  =  \boldsymbol \xi &  \textrm{in}\; {\Omega}_S\times(0,T), 
\\
&{\rho}_S \partial_{t} {\boldsymbol \xi}  =  {\nabla} \cdot \boldsymbol \sigma_S(\boldsymbol \eta) + \boldsymbol{f}_S&  \textrm{in}\; {\Omega}_S\times(0,T),  
\\
& \boldsymbol v= \boldsymbol 0 & \textrm{on} \; \partial \Omega_F / {\Gamma} \times (0,T),
\\
& \boldsymbol \eta = \boldsymbol 0 & \textrm{on} \; \partial \Omega_S / {\Gamma} \times (0,T).
\end{align*}
Using the exact solutions, we compute  forcing terms $\boldsymbol f_F, \boldsymbol f_S$ and $s$.

Implementing our methodology using finite elements was facilitated through the use of the FreeFem++  software~\cite{hecht2012new}.  For space discretization, $\mathbb{P}_1$ elements were used for both the structure velocity and displacement, where $\mathbb{P}_1$ bubble - $\mathbb{P}_1$ elements were used for the fluid velocity and pressure, respectively.  We  set  parameters $\lambda_S, \rho_S,  \mu_S, \rho_F \text{ and } \mu_F$ equal to one. The simulations were performed until the final time $T= 0.3$ s was reached. Figure~\ref{comp_actual} shows the comparison of the computed and exact fluid velocity (top)  and structure displacement (bottom) obtained with $\alpha=10$. An excellent agreement is observed.
\begin{figure}[ht]
\centering{
\includegraphics[scale=0.15]{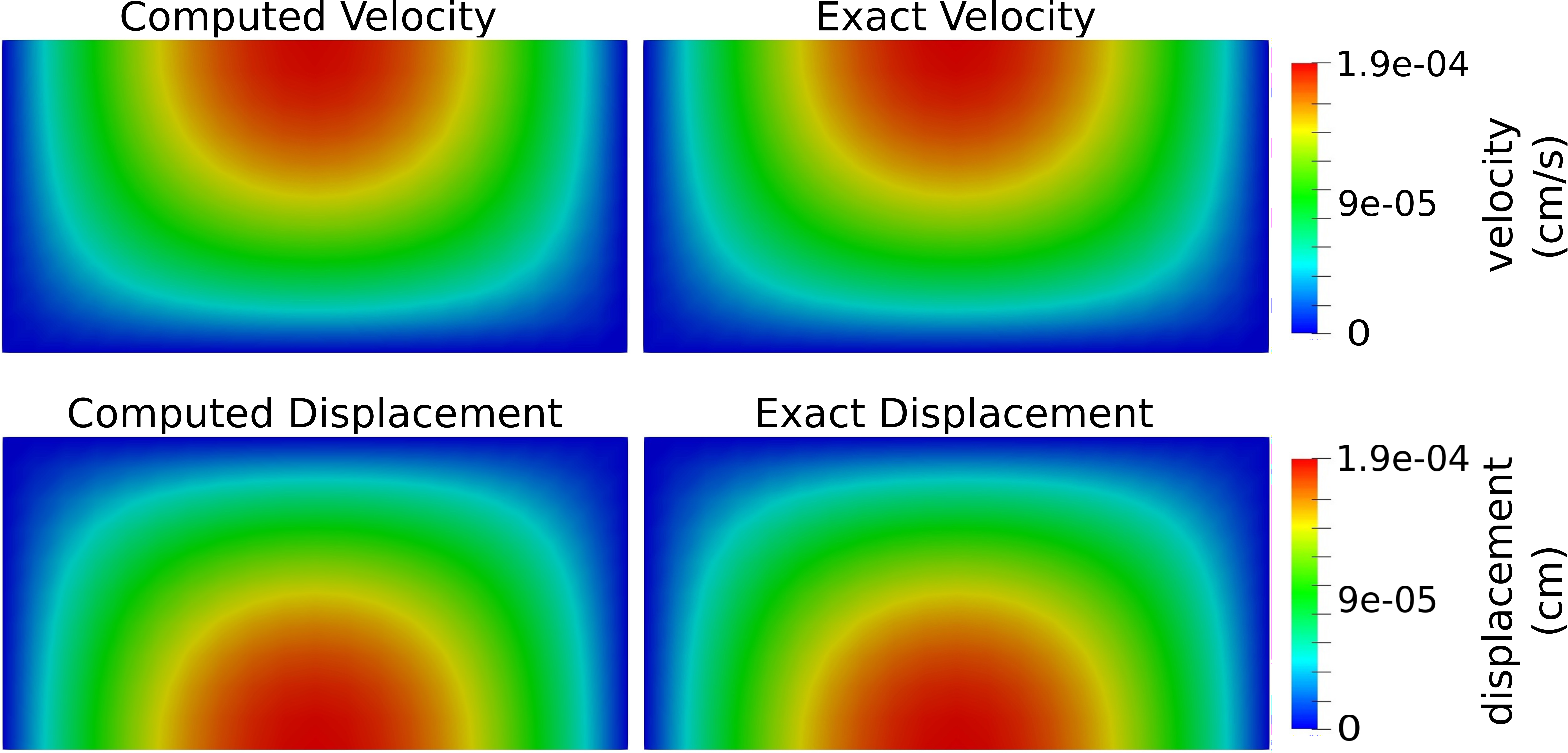}
}
\caption{Example 1: A comparison of the computed and exact fluid velocity (top) and structure displacement (bottom) at $T=0.3$ s.}
\label{comp_actual}
\end{figure}

In conjunction with comparing the numerical results to the actual solution, we compute  convergence rates as described in Theorem~\ref{MainThm} in addition to analyzing how well the coupling conditions are satisfied at the interface. In particular, we compute the following errors for the structure displacement and velocity, and fluid velocity:
\begin{align*}
e_{\boldsymbol \eta} = \frac{ \left\Vert \boldsymbol \eta -\boldsymbol \eta_{ref} \right\Vert^2_{S}}{\left\Vert \boldsymbol \eta_{ref} \right\Vert^2_S},
\quad
e_{\boldsymbol \xi}=\frac{\left\Vert \boldsymbol \xi - \boldsymbol \xi_{ref} \right\Vert_{L^2({\Omega}_S)}}{\left\Vert \boldsymbol \xi_{ref} \right\Vert_{L^2({\Omega}_S)}}, 
\quad
e_F=\frac{\left\Vert \boldsymbol v- \boldsymbol v_{ref} \right\Vert_{L^2(\Omega_F)}}{\left\Vert \boldsymbol v_{ref} \right\Vert_{L^2(\Omega_F)}},
\end{align*}
as well as  the error for the kinematic coupling condition:
\begin{align*}
e_{ke}=\frac{\left \Vert \boldsymbol v - \boldsymbol \xi \right \Vert_{\Gamma}}{\left \Vert\boldsymbol v \right\Vert_{\Gamma}},
\end{align*}
and error for the dynamic coupling condition:
\begin{align*}
e_{\boldsymbol \sigma}=\frac{\left \Vert \boldsymbol \sigma_F \boldsymbol n_F - \boldsymbol \sigma_S \boldsymbol n_F \right \Vert_{\Gamma}}{\left\Vert \boldsymbol \sigma_F \boldsymbol n_F \right \Vert_{\Gamma}}.
\end{align*}
In order to compute the convergence rates, we start  with an initial time step  $\Delta t = 0.01$ and mesh size $h=0.1$, and divide them by two for four iterations. Each variable is then evaluated with differing alphas equaling 1, 10, 100, 200, and 500.

\begin{figure}[ht]
\centering{
\includegraphics[scale=0.55]{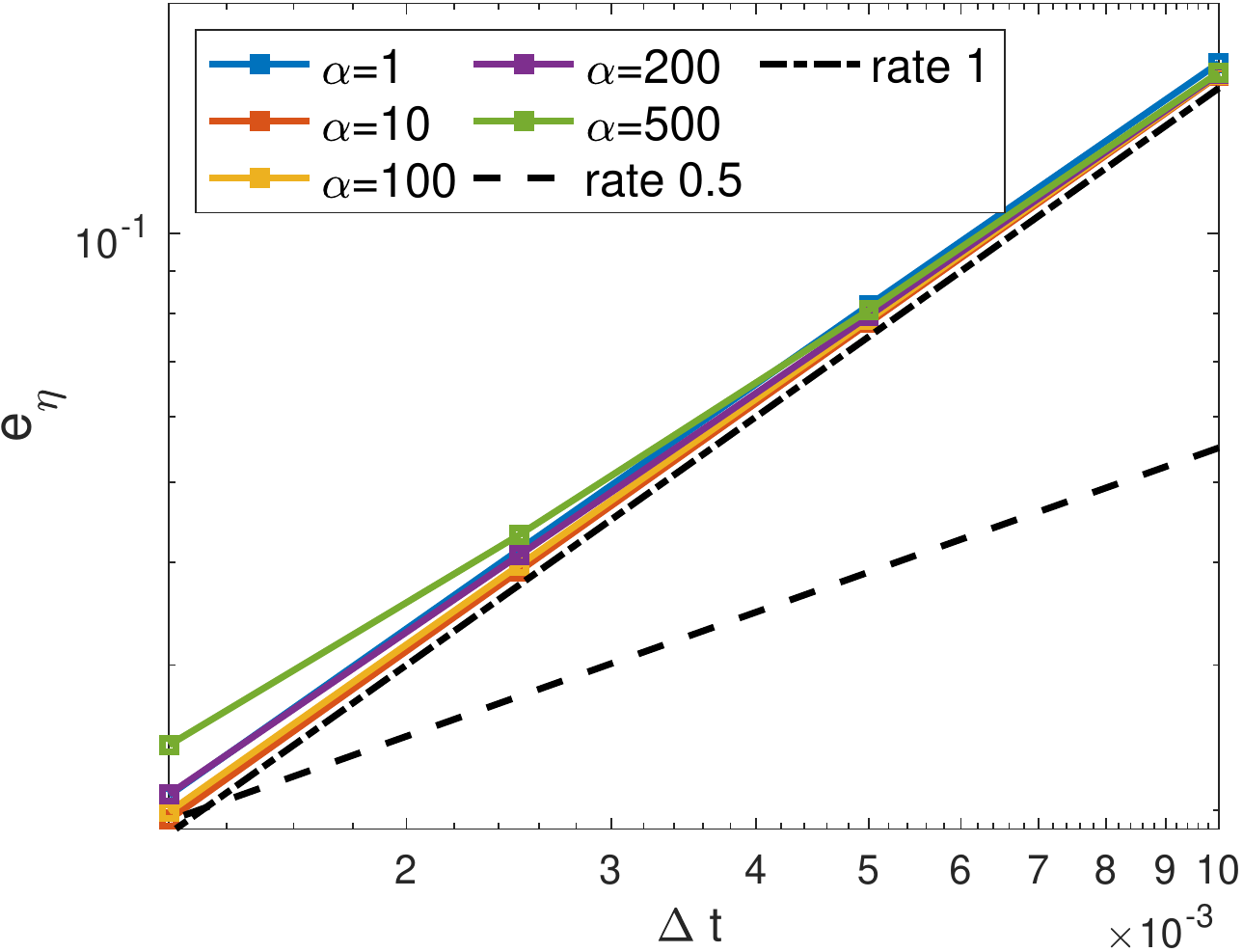}
\includegraphics[scale=0.55]{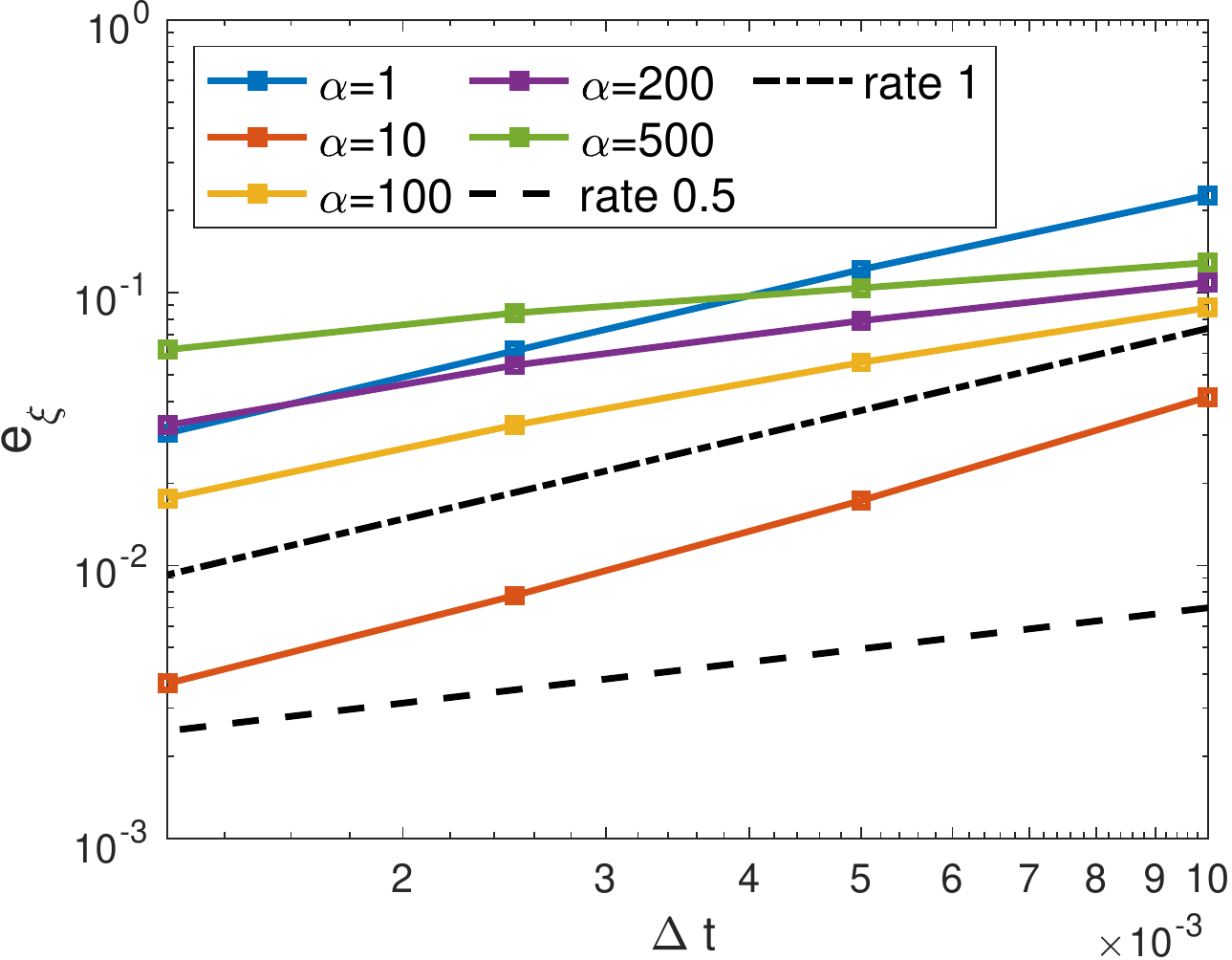}
\includegraphics[scale=0.55]{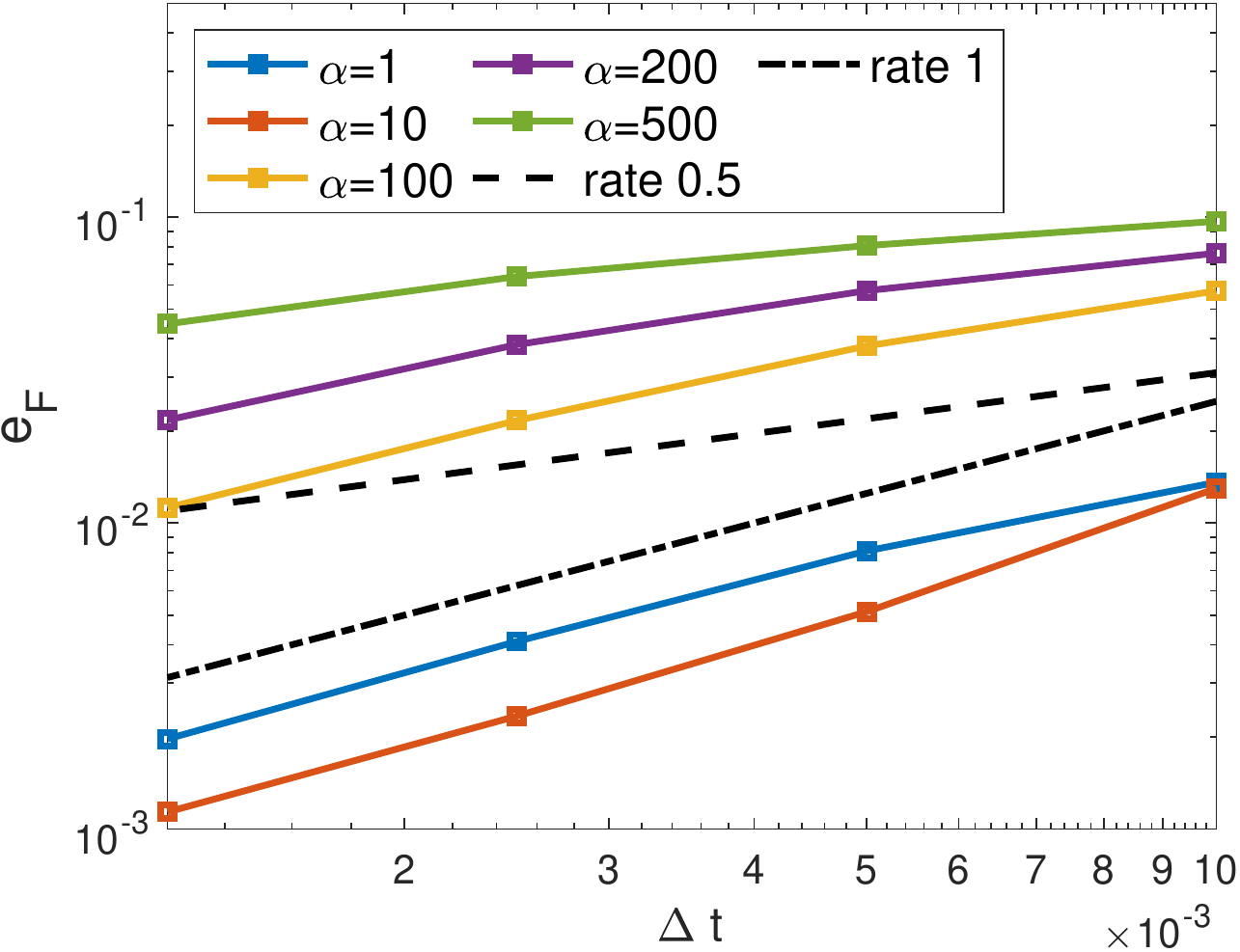}
}
\caption{Example 1: Errors  for the solid displacement $\boldsymbol \eta$ (top-left), solid velocity $\boldsymbol \xi$ (top-right), and fluid velocity $\boldsymbol v$ (bottom) at the final time $T=0.3$ s. }
\label{eta_xi_v}
\end{figure}
Figure~\ref{eta_xi_v} shows the convergence rates for the structure displacement   (top left), structure velocity   (top right) and fluid velocity   (bottom) computed at the final time.
We observe that the convergence rates for the structure displacement are  close to one across all values of $\alpha$. The convergence rates for the structure velocity are first-order, or better, when $\alpha$ is equal to 1 and 10. As $\alpha$ increases, the convergence rates begin to decrease, compromising condition~\eqref{CFLconv} used in the convergence analysis. Similar holds for the fluid velocity, which has the best convergence rates for $\alpha$ values of 10 and 100, and the worst when $\alpha$ increases to 500.

\begin{figure}[h!]
    \centering{
        \includegraphics[scale=0.55]{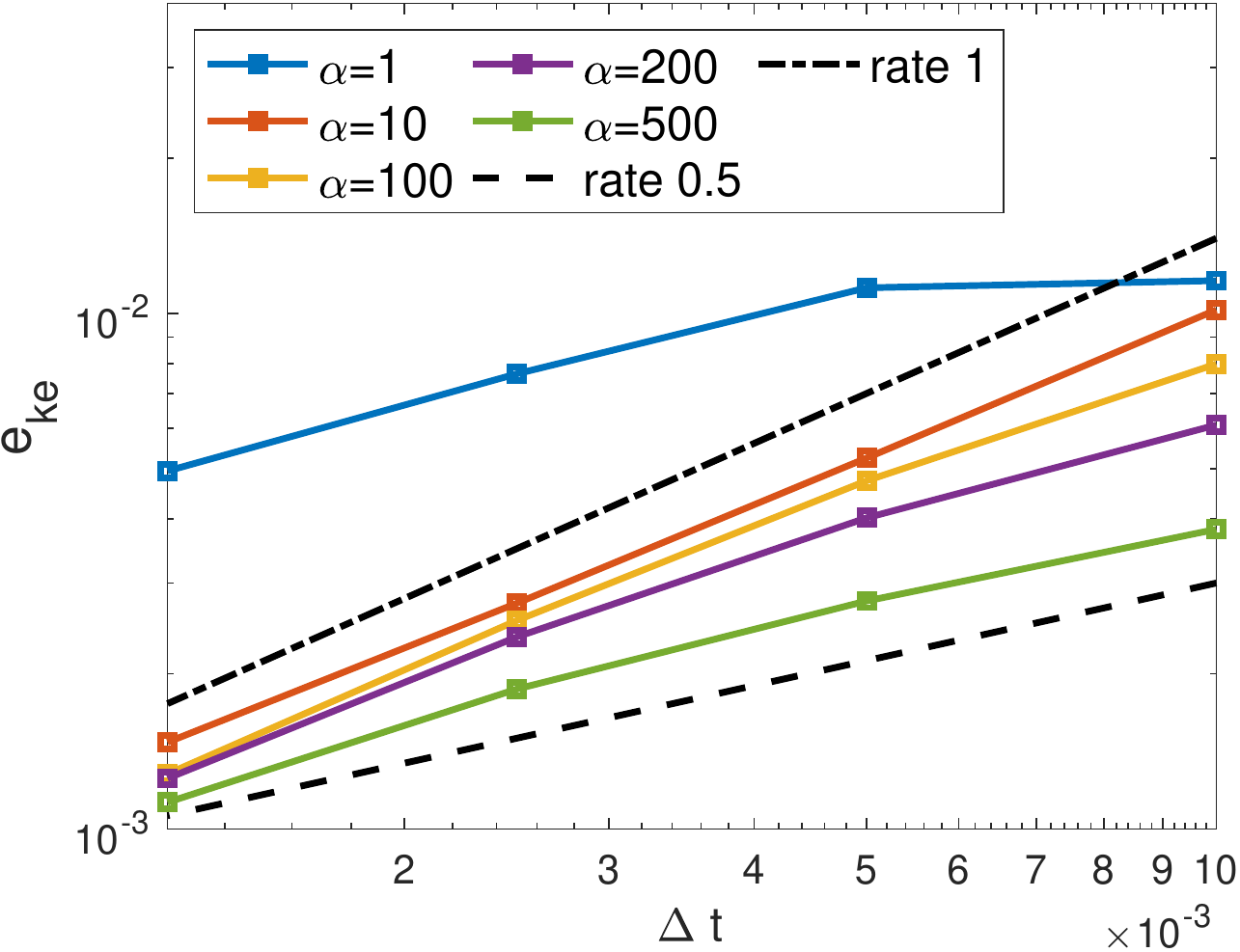}
        \includegraphics[scale=0.55]{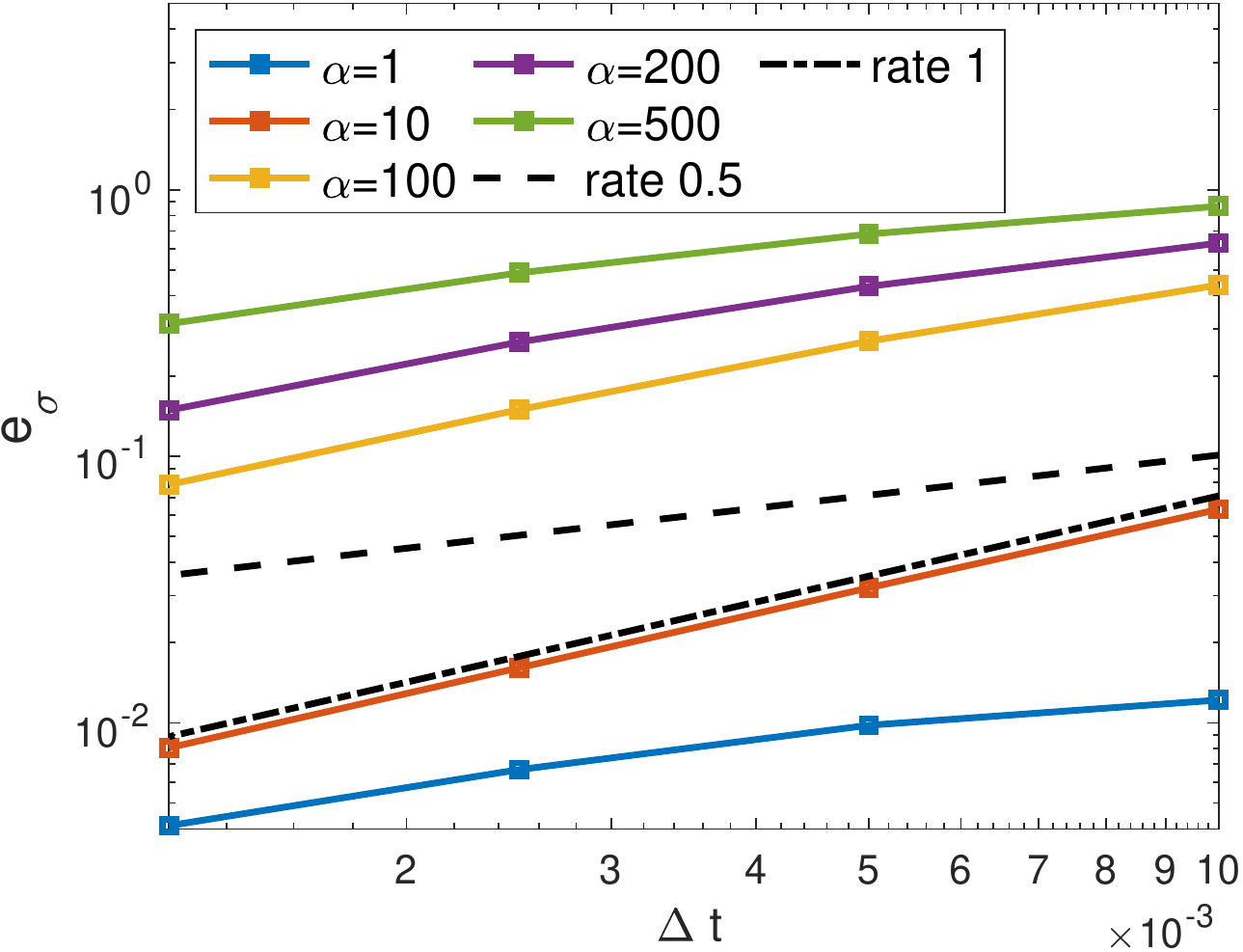}
        }
        \caption{Example 1: Kinematic (left) and dynamic (right) coupling condition errors at the final time $T=0.3$ s.}
        \label{kinematic_dynamic}
\end{figure}
In addition to the errors related to Theorem~\ref{MainThm}, we  investigate the relation between the combination parameter $\alpha$ and how well the coupling conditions are satisfied. In particular, the generalized Robin boundary condition~\eqref{lcomb} will turn into the dynamic coupling condition~\eqref{dynamic} as $\alpha \rightarrow 0$, and it will approach the kinematic coupling condition~\eqref{kinematic} as $\alpha \rightarrow \infty.$ Therefore, we compute errors $e_{ke}$ and $e_{\boldsymbol \sigma}$ as we take $\alpha=1, 10, 100, 200,$ and $500$.   In this case, to better approximate the fluid and structure stresses, we used  $\mathbb{P}_2$ elements  for fluid and structure velocities and the structure displacement, and $\mathbb{P}_1$ elements  for pressure. Figure~\ref{kinematic_dynamic} shows errors $e_{ke}$ (left) and $e_{\boldsymbol \sigma}$ (right) computed with the following time and mesh sizes:
\begin{gather}
(\Delta t, h) \in \left\{ \left( \frac{10^{-2}}{2^k}, \frac{0.0625}{2^k} \right)\right\}_{k=0}^3.
\label{discPar}
\end{gather}
We observe that, with the exception of $\alpha=1$, the convergence rates are closer to one for smaller values of $\alpha$, and they decrease to 0.5 as $\alpha$ increases to 500. We also note that the error in the kinematic coupling condition decreases as $\alpha$ increases, while the opposite holds for the dynamic coupling condition. However, for all the considered cases, the relative error in the kinematic coupling condition is significantly smaller than the relative  error in the dynamic coupling condition.

\subsection{Example 2}
In the second example, we study the accuracy of the proposed method applied to a moving domain FSI problem~\eqref{fsi1}-\eqref{fsi2}. 
We use the same manufactured solutions,~\eqref{true_eta}-\eqref{true_p}, as in Example~1. Furthermore, we  define the true solution for the fluid domain displacement to be $\boldsymbol \eta_F = \boldsymbol \eta$, and the true solution for the fluid domain velocity to be $\boldsymbol w  =\partial_t \boldsymbol \eta_F$. Similar to Example~1, we add  forcing terms to equations~\eqref{fsi1},~\eqref{fsi11} and~\eqref{fsi12}. 
To update the fluid domain, we solve 
\begin{align*}
&- \Delta {\boldsymbol \eta}^{n+1}_F = \boldsymbol f_D & \textrm{in} \; \hat{\Omega}_F,  \\
&  {\boldsymbol \eta}^{n+1}_F = 0 & \textrm{on} \; \hat{\Gamma}^{in}_F \cup \hat{\Gamma}^{out}_F, \\
&  {\boldsymbol \eta}^{n+1}_F = {\boldsymbol \eta}^{n+1} & \textrm{on} \; \hat{\Gamma}.
\end{align*}
As for $\boldsymbol f_F, \boldsymbol f_S$ and $s$, we compute $\boldsymbol f_D$ using the exact solution.
Every other aspect of this example remains unchanged, meaning the error calculations, space and time discretization specifications, and parameters are the same as in Example~1.

\begin{figure}[ht]
\centering{
\includegraphics[scale=0.55]{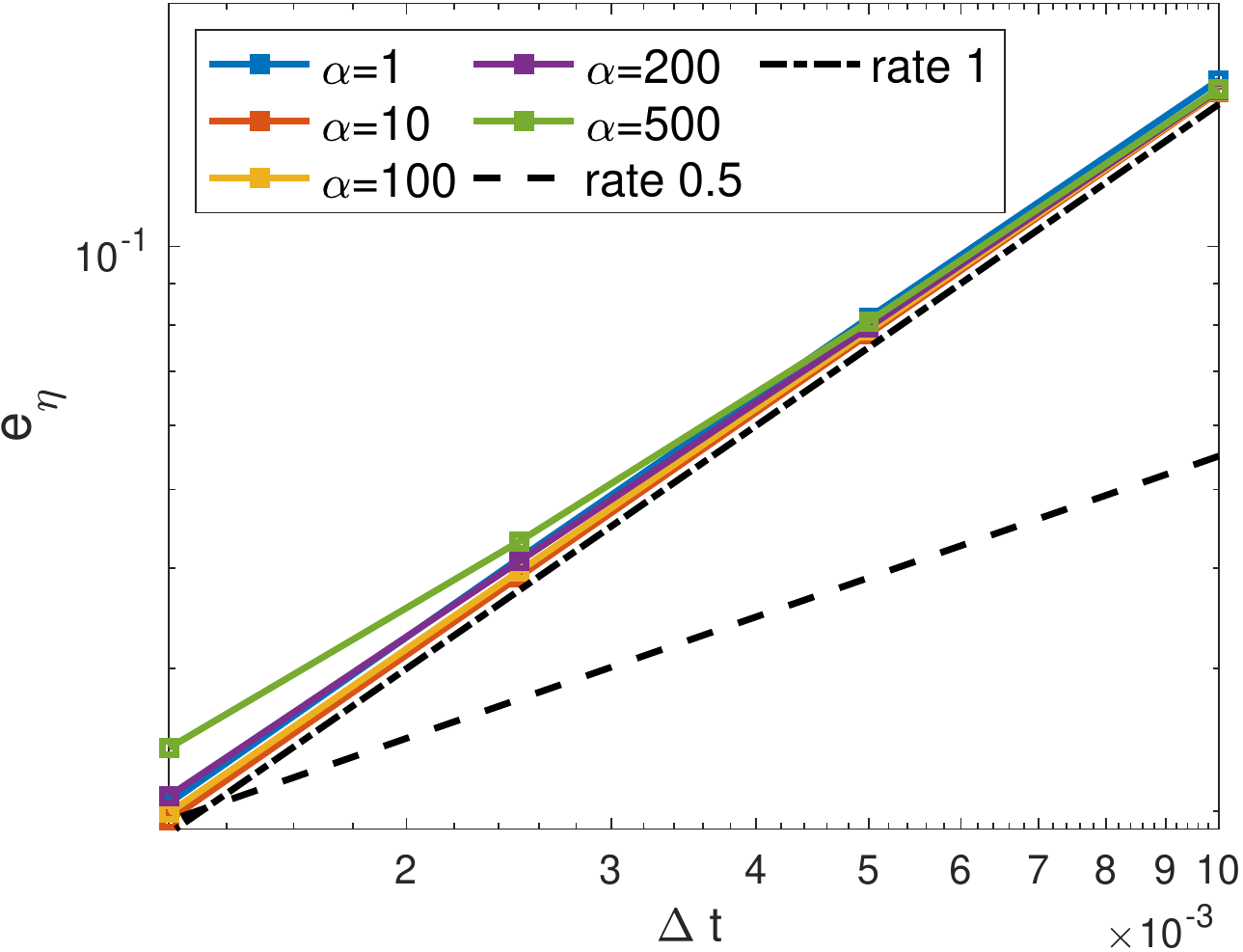}
\includegraphics[scale=0.55]{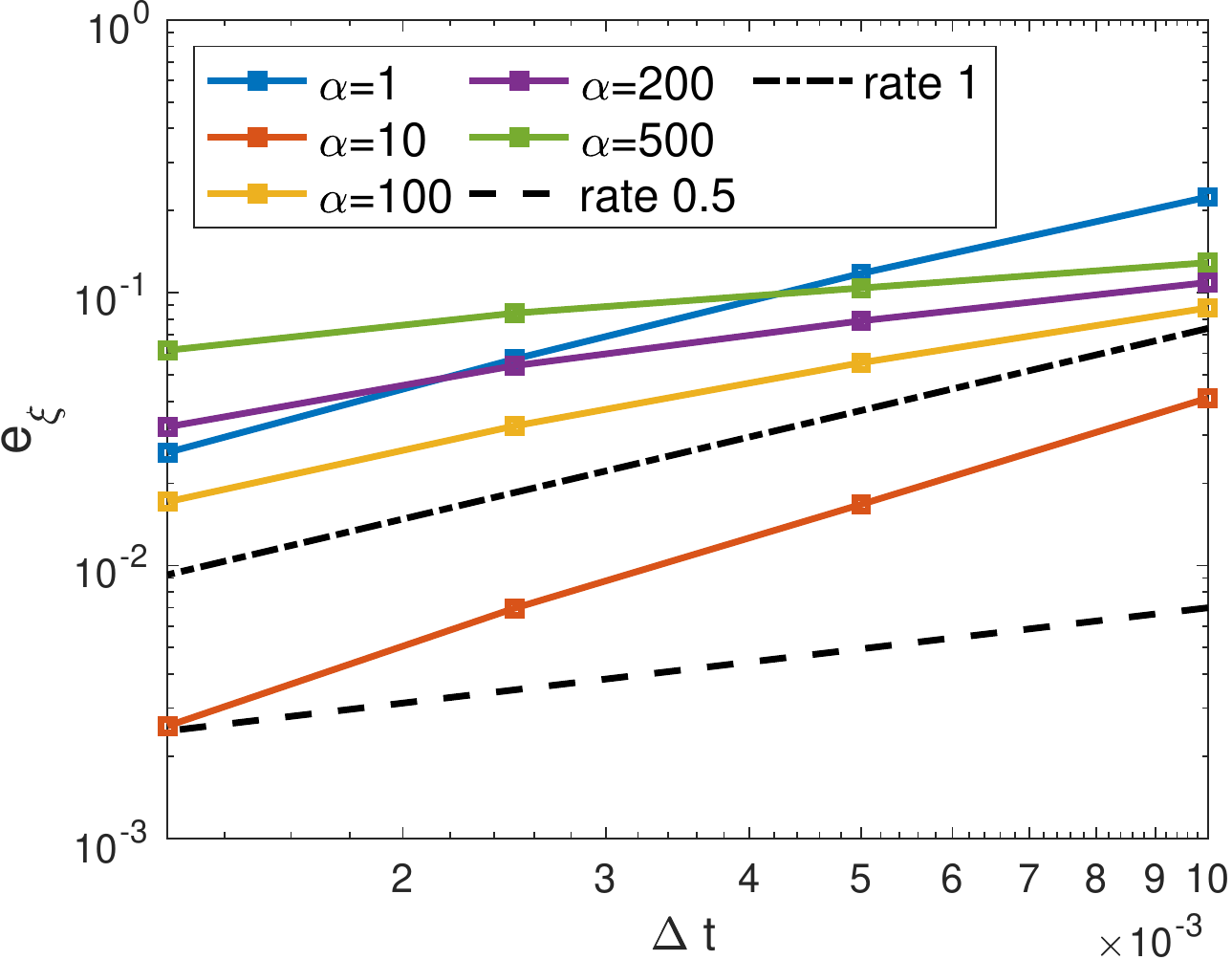}
\includegraphics[scale=0.55]{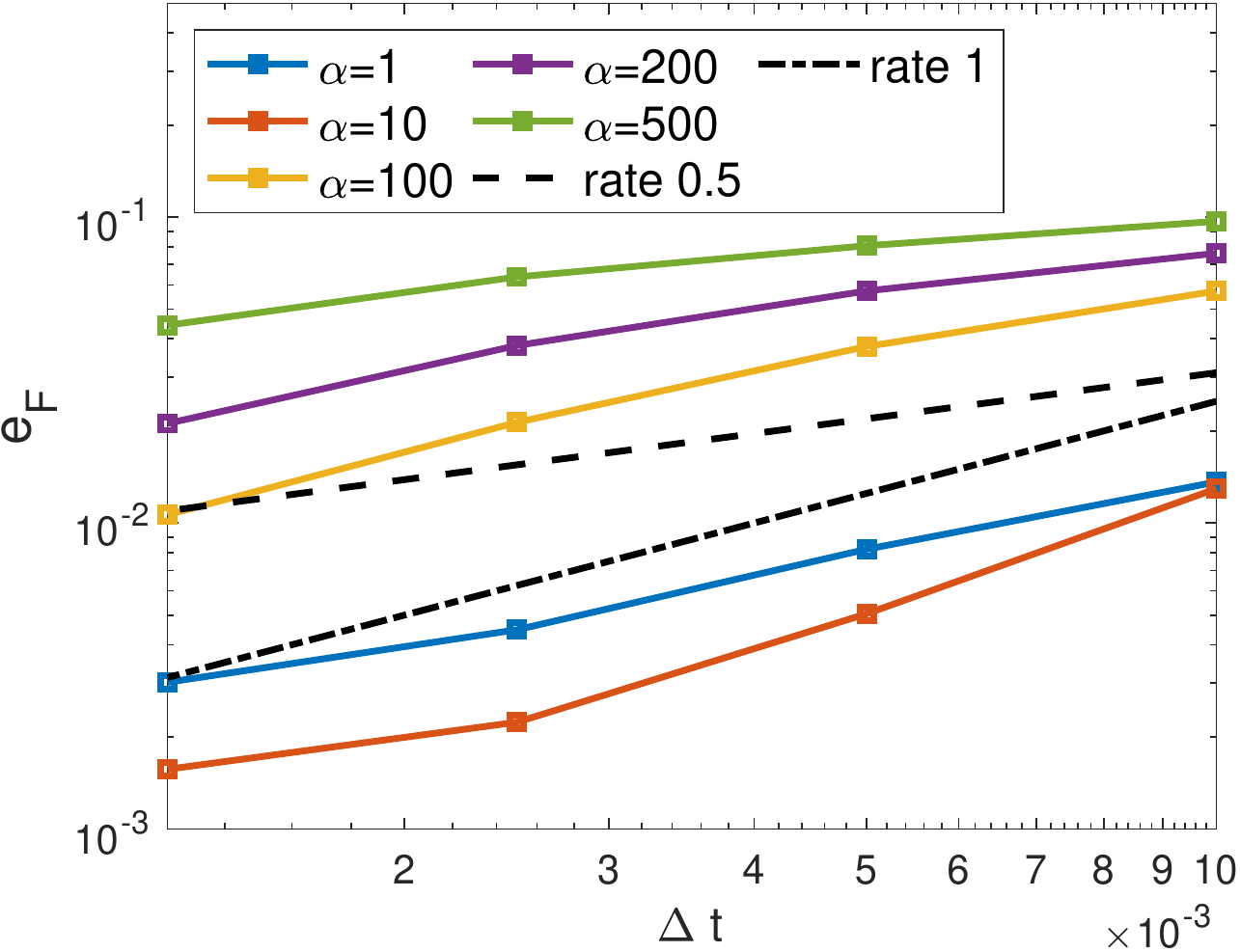}
}
\caption{Example 2: Errors  for the solid displacement $\boldsymbol \eta$ (top-left), solid velocity $\boldsymbol \xi$ (top-right), and fluid velocity $\boldsymbol v$ (bottom) at the final time $T=0.3$ s.}
\label{eta_xi_v2}
\end{figure}
Figure~\ref{eta_xi_v2} shows the errors for the structure displacement (top left), structure velocity (top right) and fluid velocity (bottom) obtained at $T=0.3$ s. Similar behavior is observed as in Example 1. For all values of $\alpha$, the convergence rates for the solid displacement are close to one, while the errors are roughly the same with the very slight exception of when $\alpha=500$. The convergence rates for solid velocity decrease from one to 0.5 as the values of $\alpha$ increase, while the errors themselves grow as $\alpha$ increases with the exception of $\alpha = 1$.  In a similar trend, the rates for the fluid velocity decrease and the errors increase
 as $\alpha$ grows, with the exception of $\alpha=1$. For all variables, the best convergence rates and the smallest errors are obtained with $\alpha=10$.

\begin{figure}[h!]
    \centering{
        \includegraphics[scale=0.55]{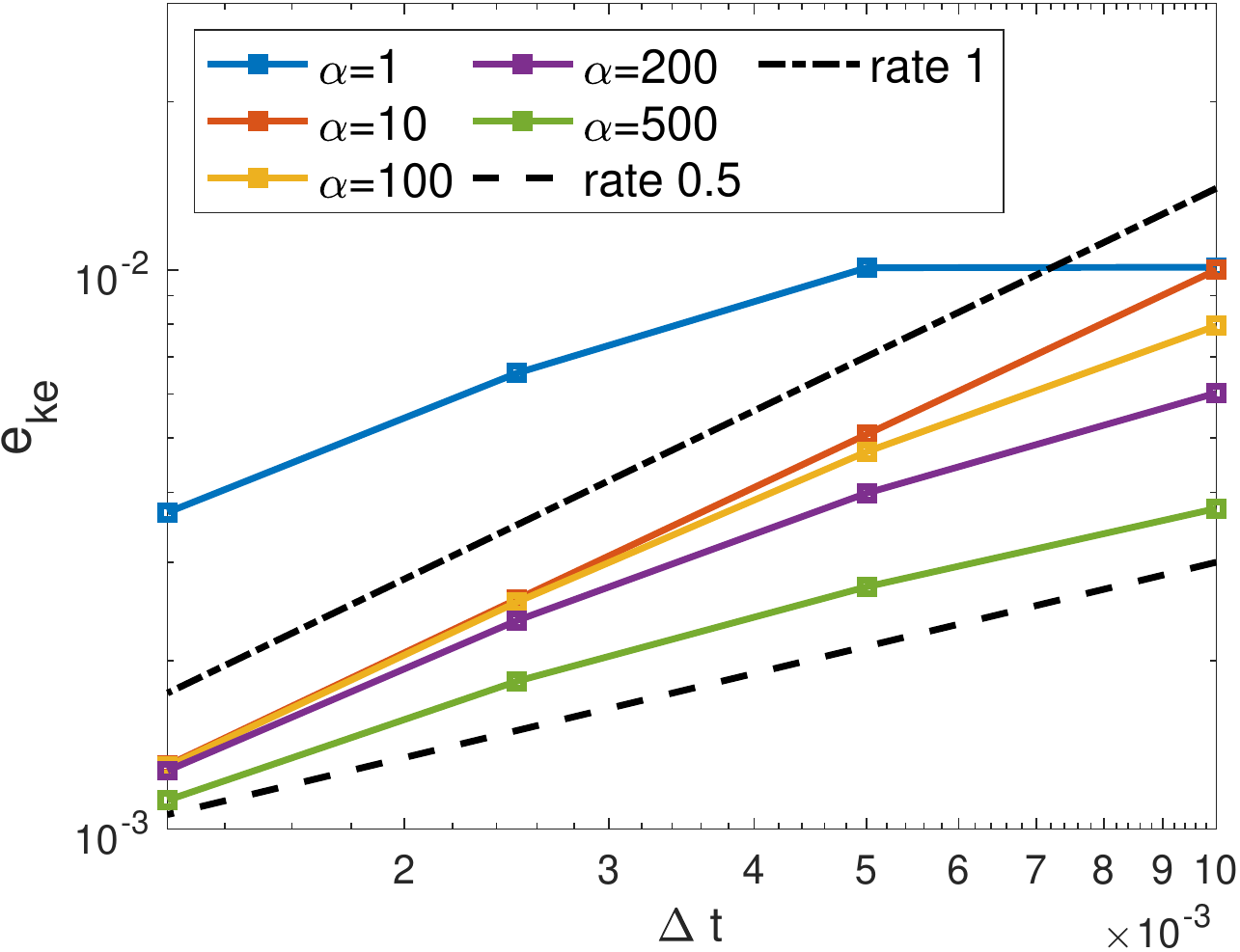}
        \includegraphics[scale=0.55]{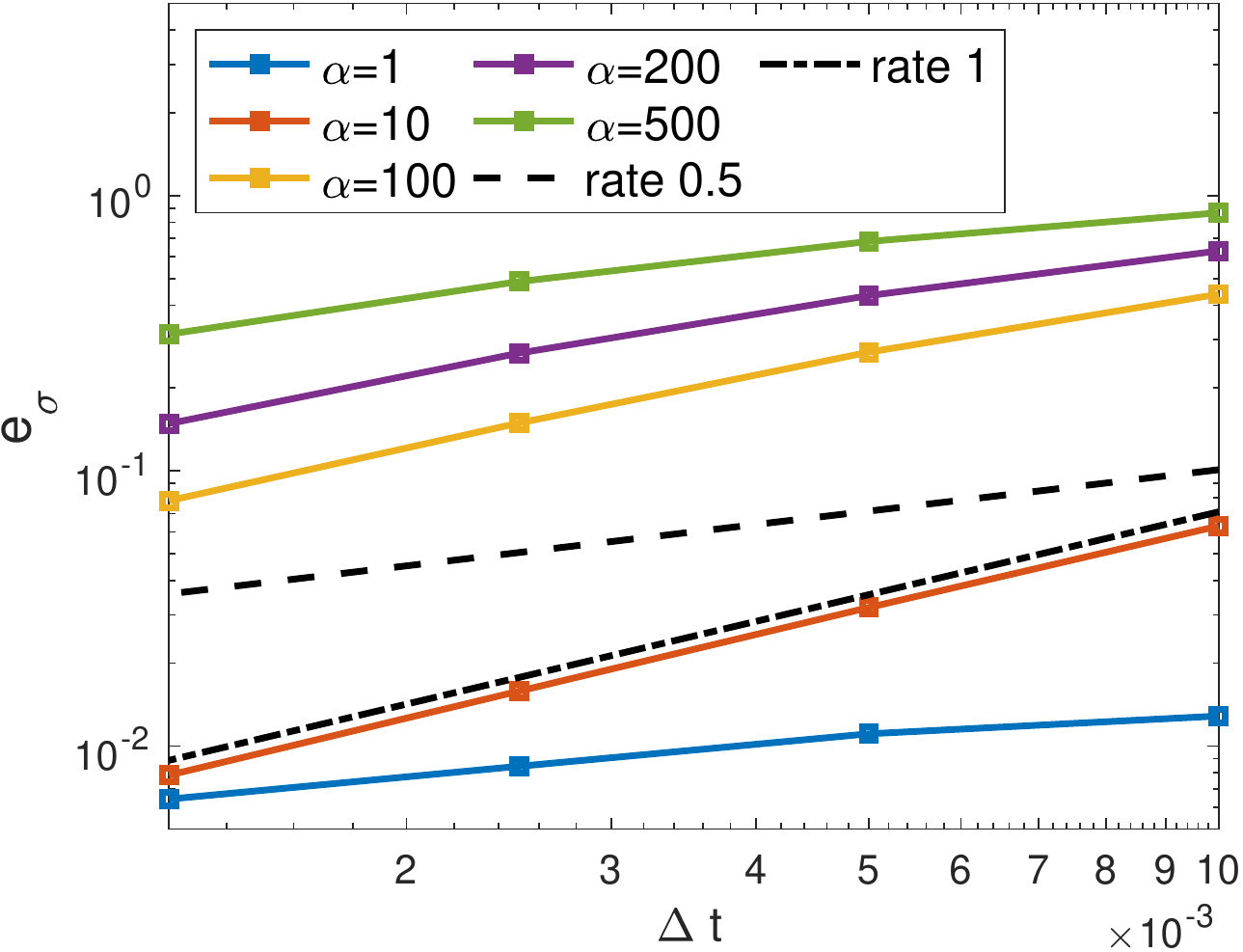}
        }
        \caption{Example 2: Kinematic (left) and dynamic (right) coupling condition errors at the final time $T=0.3$ s.}
        \label{kinematic_dynamic2}
\end{figure}
Likewise to Example~1, we calculate the errors in approximating coupling conditions using a $\mathbb{P}_1$ space discretization for  pressure and $\mathbb{P}_2$ for all other variables. The temporal and spatial discretization parameters are the same as described in~\eqref{discPar}. Figure~\ref{kinematic_dynamic2} shows the kinematic coupling condition error (left) and the dynamic coupling condition error (right) at $T=0.3$ s obtained using different values of $\alpha$.
Similar to what we observed in Example~1, as $\alpha$ increases, the error decreases for the kinematic coupling condition with the reversed result for the dynamic coupling condition. As for convergence rates, we obtain  values around 0.5 using $\alpha=1$ and values very close to one using $\alpha=10$, which then decrease back down to 0.5 as $\alpha$ increases.

\subsection{Example 3}

The third example focuses on a classical benchmark problem used in the validation of FSI solvers~\cite{bukavc2014modular}. We consider the fluid flow in a two-dimensional channel interacting with a deformable wall.  The  reference fluid and structure domains are defined as $\hat{\Omega}_F = (0,6) \times (0,0.5)$ and $\hat\Omega_S= (0,6) \times (0.5,0.6)$, respectively. We consider the moving domain FSI problem~\eqref{fsi1}-\eqref{fsi2},  where we add a linearly elastic spring term, $\gamma \boldsymbol \eta$, to the elastodynamic equation, yielding:
\begin{align*}
\rho_S \partial_t \boldsymbol \xi + \gamma \boldsymbol \eta = \nabla \cdot \boldsymbol \sigma_S(\boldsymbol \eta)
\qquad \textrm{in} \; \hat{\Omega}_S \times (0,T).
\end{align*}
Term $\gamma \boldsymbol \eta$ is obtained from
the axially symmetric model, and it represents a spring keeping the top and
bottom boundaries in a two-dimensional model connected~\cite{bukavc2014modular}.

The parameters used in this example, $\rho_F$ = 1 g/cm$^3, \mu_F$ = 0.035 g/cm s$, \rho_S=1.1$ g/cm$^3, \mu_S=5.75 \cdot 10^5$ dyne/cm$^2$, $ \gamma=4 \cdot 10^6$  dyne/cm$^4$  and $\lambda_S=1.7 \cdot 10^6$  dyne/cm$^2$, are within physiologically realistic values of blood flow in compliant arteries.  In this example, we use $\alpha=100$.  
The flow is driven by prescribing a time-dependent pressure drop at the inlet and outlet sections, as defined in~\eqref{inlet}-\eqref{outlet}, where
\begin{align}
    p_{in}(t)=\left\{
                \begin{array}{ll}
                  \frac{p_{max}}{2} \left[1-\cos \left(\frac{2 \pi t}{t_{max}}\right) \right],    & \text{if } t \leq t_{max}\\
                  0,    &\text{if }t > t_{max}
                \end{array}
              \right.
             ,\text{        }p_{out}=0 \text{ } &\forall t \in (0,T).
\end{align}
The pressure pulse is in effect for $t_{max} = 0.03$ s with maximum pressure $p_{max}=1.333 \times 10^4$ dyne/$\text{cm}^2$.  The final time is $T=12$ ms.
We use $\mathbb{P}_1$ bubble - $\mathbb{P}_1$ elements   for the fluid velocity and pressure, respectively, and  $\mathbb{P}_1$ elements for  the structure velocity and displacement. The results are obtained using $\Delta t=10^{-5}$ on a mesh containing 7,500 elements in the fluid domain and 1,200 elements in the structure domain. 

\begin{figure}[h!]
    \centering{
        \includegraphics[scale=0.55]{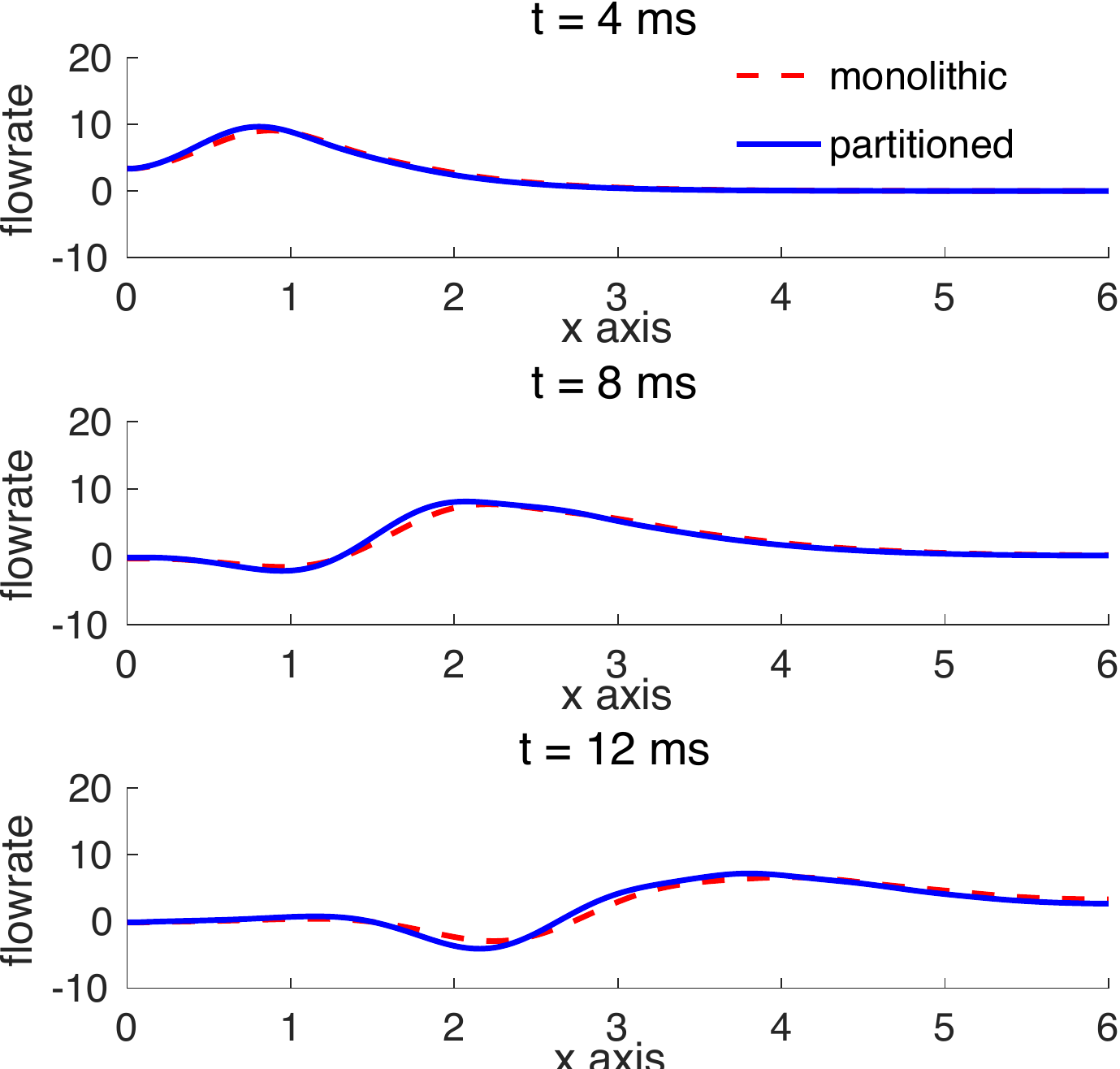}
        }
        \caption{Fluid flowrate vs. x-axis compared with a monolithic scheme.}
        \label{monolithiccomparisonF}
\end{figure}
\begin{figure}[h!]
    \centering{
        \includegraphics[scale=0.55]{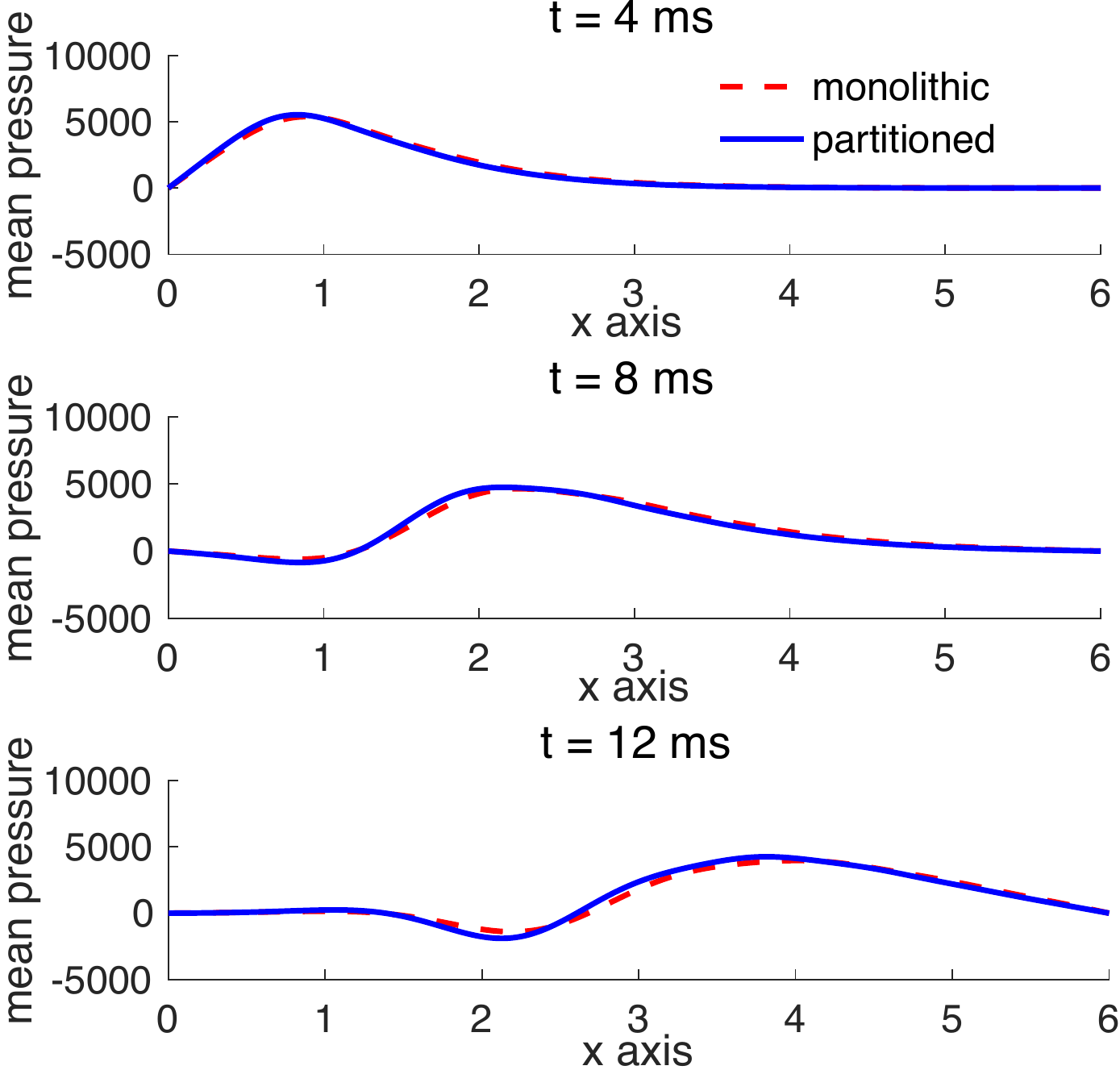}
        }
        \caption{Fluid pressure vs. x-axis compared with a monolithic scheme.}
        \label{monolithiccomparisonP}
\end{figure}

\begin{figure}[h!]
    \centering{
        \includegraphics[scale=0.55]{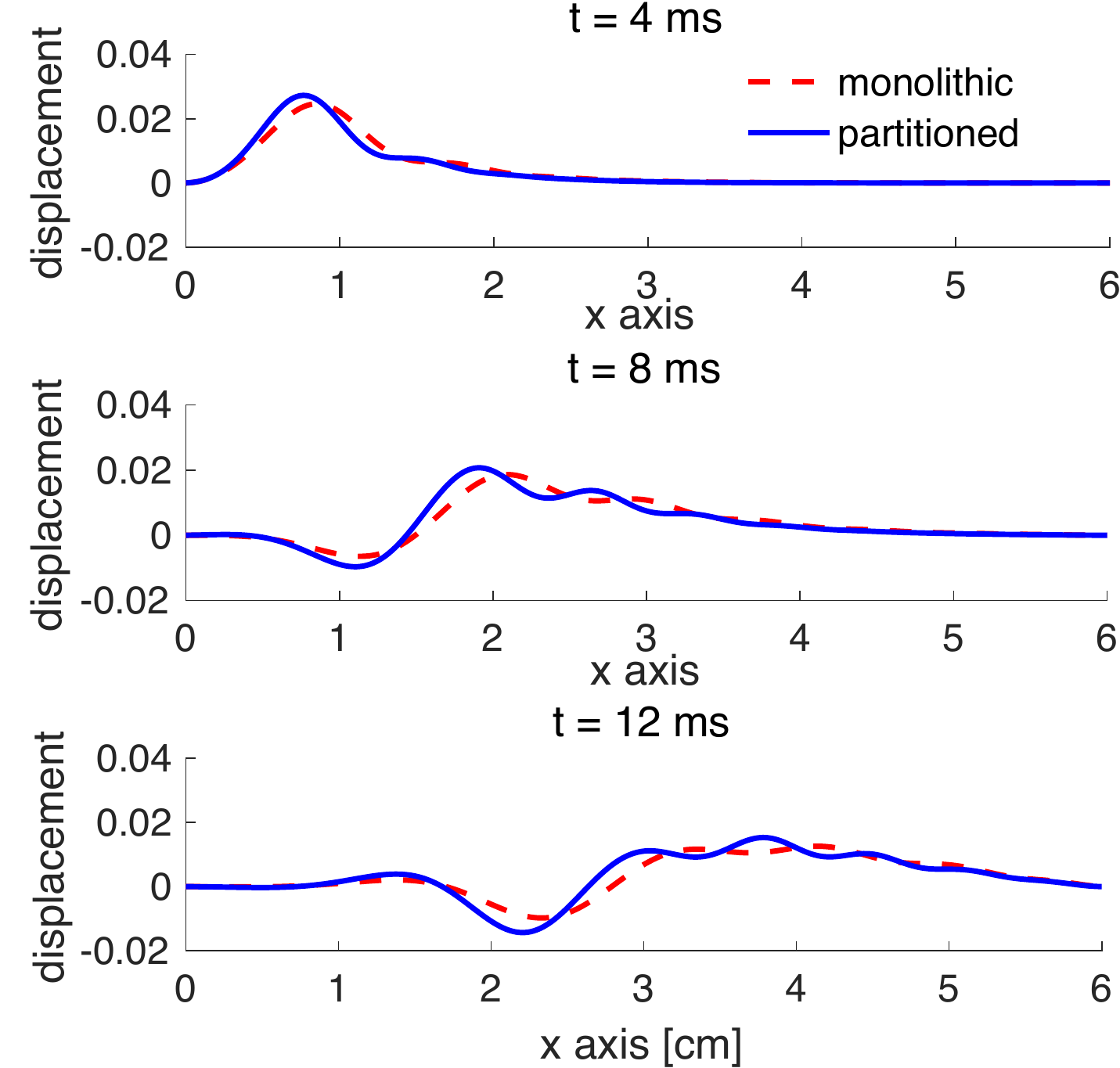}
        }
        \caption{Fluid-structure interface displacement vs. x-axis compared with a monolithic scheme.}
        \label{monolithiccomparisonD}
\end{figure}

Figures~\ref{monolithiccomparisonF},~\ref{monolithiccomparisonP} and~\ref{monolithiccomparisonD} show a comparison of the flowrate, mean pressure and fluid-structure interface displacement obtained using the proposed numerical method and a monolithic scheme used in~\cite{bukavc2014modular,quaini2009algorithms} at times $t=4, 8,$ and 12 ms. A good agreement is observed in all cases, even with small discrepancies in the interface displacement. We note that the time step used in the simulations obtained with a monolithic solver is $\Delta t=10^{-4}$. As expected, due to the splitting error, a smaller time-step was needed in the partitioned scheme.

\section{Conclusions}
We present a novel partitioned, non-iterative method for FSI problems with thick structures. The presented method is based on generalized Robin boundary conditions, which are designed by linearly combining kinematic and dynamic coupling conditions using a combination parameter, $\alpha$. 
Thanks to a novel design of Robin boundary conditions used in the fluid and structure sub-problems, we prove unconditional stability of the semi-discrete numerical method applied to a  moving domain FSI problem.  Convergence analysis was performed for a fully-discrete, linearized problem, yielding $\mathcal{O} (\Delta t^{\frac12})$ accuracy in time and optimal accuracy in space.  
The theoretically obtained results are verified in numerical examples.   In particular, using the method of manufactured solutions, we compute the relative errors between the numerical and exact solutions on both fixed domain and moving domain problems. In particular, we compute the convergence rates for different values of the combination parameter $\alpha$, and note that increasing values of $\alpha$ will lead to a decrease of convergence rates from one to 0.5 for a fixed $\Delta t$. We also compare our results to the ones obtained using a monolithic scheme on a benchmark problem of pressure propagation in a two-dimensional channel,  obtaining a good agreement. However, due to the splitting error and sub-optimal accuracy, a smaller time step was used in the partitioned scheme. 
 An extension of the proposed method to higher-order accuracy will be considered in our future work. 
 
 {\color{black}
 One of the drawbacks of the proposed method is its dependence on the combination parameter $\alpha$,  
which is, generally, problem dependent. In other work where similar combination parameters are introduced, such as~\cite{gerardo2010analysis}, the authors suggest to use
\begin{gather}
\alpha = \frac{\rho_S H_S}{\Delta t} + \beta H_S \Delta t, 
\label{alphaformula}
\end{gather}
where $H_S$ is the height of the solid domain and 
$$\beta=\frac{E}{1-\nu^2}(4\rho_1^2 - 2(1-\nu)\rho_2^2),$$
with $E$ denoting the Young's modulus, $\nu$ denoting the Poisson's ratio and  $\rho_1$ and $\rho_2$ denoting the mean and Gaussian curvatures of the fluid-structure interface, respectively. However, this choice of $\alpha$ is proposed to ensure convergence of a subiterative solution procedure when solving strongly coupled  FSI problems. Since we do not need subiterations to achieve stability, we do not require similar conditions on $\alpha$. Indeed,   using~\eqref{alphaformula} to compute $\alpha$ in our method gives results that are not optimally accurate.  Therefore, $\alpha$ needs to be estimated separately for each  problem.

}

\section{Acknowledgments} 
This work was partially supported by NSF under grants DMS
 1619993 and  1912908, and DCSD 1934300. We would  like to thank Prof. Catalin Trenchea for helpful
discussions and suggestions.

\bibliographystyle{ieeetr}
\bibliography{FSIthickLinearMovingR1}
\end{document}